\documentclass[11pt,reqno]{amsart}
\usepackage{a4wide}
\numberwithin{equation}{section}
\usepackage{mathrsfs}
\usepackage{amsfonts}
\usepackage{amsmath}
\usepackage{stmaryrd}
\usepackage{amssymb}
\usepackage{amsthm}
\usepackage{mathrsfs}
\usepackage{url}
\usepackage{amsfonts}
\usepackage{amscd}
\usepackage{indentfirst}
\usepackage{enumerate}
\usepackage{amsmath,amsfonts,amssymb,amsthm}
\usepackage{amsmath,amssymb,amsthm,amscd}
\usepackage{graphicx,mathrsfs}
\usepackage{appendix}

\usepackage[numbers,sort&compress]{natbib}
\usepackage{color}
 \usepackage[colorlinks, linkcolor=blue, citecolor=blue]{hyperref}

\newcommand{\R}{\mathbb{R}}

\renewcommand{\theequation}{\arabic{section}.\arabic{equation}}

\newtheorem{Thm}{Theorem}[section]
\newtheorem{Lem}[Thm]{Lemma}

\newtheorem{Prop}[Thm]{Proposition}

\newtheorem{Rem}[Thm]{Remark}

\begin{document}
\title[{Normalized solutions to fractional Schr\"odinger equations}]
{Concentrated solutions to fractional Schr\"odinger equations with prescribed $L^2$-norm
}

\author[Q. Guo, P. Luo, C. Wang, J. Yang]{Qing Guo, Peng Luo, Chunhua Wang and Jing Yang}

\address[Qing Guo]{College of Science, Minzu University of China, Beijing 100081, China}
\email{guoqing0117@163.com}

\address[Peng Luo]{School of Mathematics and Statistics, Central China Normal University, Wuhan 430079, China}
\email{pluo@mail.ccnu.edu.cn}

\address[Chunhua Wang]{School of Mathematics and Statistics and Hubei Key Laboratory of Mathematical Sciences, Central China Normal University, Wuhan 430079, China}
\email{chunhuawang@mail.ccnu.edu.cn}

\address[Jing Yang]{School of Science, Jiangsu University of Science and Technology, Zhenjiang 212003, China}
\email{yyangecho@163.com}


\date{\today}


\begin{abstract}
In this work, we investigate the existence and local uniqueness of normalized $k$-peak solutions for the fractional Schr\"odinger equations
with attractive interactions with a class of  degenerated trapping potential with
non-isolated critical points.

Precisely, applying the finite dimensional reduction method, we first obtain the existence of $k$-peak concentrated solutions
and especially describe the relationship between the chemical potential $\mu$ and the attractive interaction $a$.
Second, after precise analysis of the concentrated points and the Lagrange multiplier, we prove the
local uniqueness of the $k$-peak solutions with prescribed $L^2$-norm, by use of the local Pohozaev identities, the blow-up analysis and the maximum principle associated to the nonlocal operator $(-\Delta)^s$.

To our best knowledge, there is few results on the excited normalized solutions of the fractional Schr\"odinger equations before this present work. The main difficulty lies in the non-local property of the operator $(-\Delta)^s$. First, it makes the standard comparison argument in the ODE theory invalid to use in our analysis. Second,
because of the algebraic decay involving the approximate solutions, the estimates, on the Lagrange multiplier for example,  would become more subtle. Moreover,
 when studying the corresponding harmonic extension problem, several local Pohozaev identities are constructed and  we have to
estimate several kinds of integrals that never appear in the classic local Schr\"odinger problems.
In addition, throughout our discussion, we need to distinct the different cases of $p-1<\frac{4s}N$, $p-1=\frac{4s}N$, and $p-1>\frac{4s}N$,
which are called respectively that the mass-subcritical, the mass-critical, and the mass-supercritical case,  due to the mass-constraint condition.
Another difficulty comes from the influence of
the different degenerate rates along different
directions at the critical points of the potential $V(x).$
\end{abstract}

\maketitle

\section{Introduction and main results}
\setcounter{equation}{0}
In this paper, we consider the following  fractional Schr\"odinger equation
\begin{equation}\label{eq1}
\begin{cases}
(-\Delta)^s u+V(x)u= au^{p}+\mu u,  &{x\in \R^N},\\
u\in H^s(\R^N),
\end{cases}
\end{equation}
under the mass constraint
\begin{equation}\label{eqnorm}
\int_{\R^N}u^2(x)dx=1,
\end{equation}
where $s\in(0,1)$, $p\in(1,2_s^*-1)$ with
 $2_s^*=\frac{2N}{N-2s}$ is the fractional critical Sobolev exponent.

The fractional Laplacian operator, appearing in many areas including biological modeling, physics and mathematical finances,
can be regarded as the infinitesimal generator of a stable Levy process \cite{applebaum}. For any $s\in(0,1)$, $(-\Delta)^s$ is the nonlocal operator defined as
\begin{align*}
(-\Delta)^su=c(N,s)P.V.\int_{\R^N}\frac{u(x)-u(y)}{|x-y|^{N+2s}}dy,
\end{align*}
where $P.V.$ is the principal value and $c(N,s)=\pi^{2s+\frac N2}\Gamma(s+\frac N2)/\Gamma(-s)$. One could refer to \cite{chen-li-ma,npv} for more details on the fractional Laplacian operator.
Particularly, this nonlocal operator $(-\Delta)^s $ in $\R^N$ can be expressed as a generalized Dirichlet-to-Neumann map for a certain elliptic boundary value problem
with local differential operators defined on the upper half-space $\R^{N+1}_+=\{(x,t): x\in\R^N,t>0\}$.
Precisely, by \cite{cs1},
for any $u\in\dot H^s(\R^N)$, set
\begin{align*}
\tilde u(x,t)=\mathcal P_s[u]=\int_{\R^N}\mathcal P_s(x-z,t)u(z)dz,\ \ (x,t)\in\R^{N+1}_+,
\end{align*}
where $$\mathcal P_s(x,t)=\beta(N,s)\frac{t^{2s}}{(|x|^2+t^2)^{\frac{N+2s}2}}$$
with a constant $\beta(N,s)$ such that $\displaystyle\int_{\R^N}\mathcal P_s(x,1)dx=1$.
Then $\tilde u\in L^2(t^{1-2s},K)$ for any compact set $K$ in $\overline{\R^{N+1}_+}$, $\nabla\tilde u\in L^2(t^{1-2s},\R^{N+1}_+)$
and $\tilde u\in C^\infty(\R^{N+1}_+)$. Moreover, $\tilde u$ satisfies
\begin{equation*}
\begin{cases}
div(t^{1-2s}\nabla\tilde u)=0,  &{x\in\R^{N+1}_+},\\[1mm]
-\displaystyle\lim_{t\rightarrow0}t^{1-2s}\partial_t\tilde u(x,t)=\omega_s(-\Delta u)^su(x),   &{x\in\R^{N}},
\end{cases}
\end{equation*}
in the distribution sense, where $\omega_s=2^{1-2s}\Gamma(1-s)/\Gamma(s)$. Moreover, it holds that
\begin{align*}
\|\tilde u\|_{L^2(t^{1-2s},\R^{N+1}_+)}=\omega_s\|u\|_{\dot H^s}.
\end{align*}
Without loss of generality, we may assume $\omega_s=1$.

Problems with fractional Laplacian have been extensively studied recently, see for example \cite{bcps,b2cps,bcss,cabresire,ct1,cs1,dpv,fqt1,fl1,fls1,gn1,gln,jlx,ntw1,tan,tx1,silvestre,sv1,weinstein,yyy} and the references therein.
In particular, the existence of multiple spike solution involving the stable critical points of the potential $V(x)$ was considered in \cite{davilapinowei}, where the critical points of $V(x)$ seem to be isolated.

\medskip

We first recall the well-known results about the ground state of the following equation
\begin{equation}\label{eqgs}
(-\Delta)^s u+u= u^{p},\ \ u>0,\ \ x\in \R^N,\ \ u(0)=\max_{x\in\R^N}u(x).
\end{equation}
Let $N\geq1,s\in(0,1)$ and $1<p<2^*_s-1$. Then the following hold (c.f.\cite{fls1,fl1}).\medskip\\
(i)(Uniqueness) The ground state solution $U\in H^s(\R^N)$ of \eqref{eqgs} is unique.\smallskip\\
(ii)(Symmetry, regularity, and decay) $U(x)$ is radial, positive, and strictly decreasing in $|x|$. Moreover, $U\in H^{2s+1}(\R^N)\cap C^\infty(\R^N)$ and
satisfies
\begin{equation*}
\frac{C_1}{1+|x|^{N+2s}}\leq U(x)\leq\frac{C_2}{1+|x|^{N+2s}},\ \ for\,\, x\in \R^N,
\end{equation*}
with some constants $C_2\geq C_1>0$.\smallskip\\
(iii)(Non-degeneracy) The linearized operator $L_0=(-\Delta)^s+1-pU^{p-1}$ is non-degenerate, i.e. its kernel is given by
$$ker L_0=span\big\{\partial_{x_1}U,\partial_{x_2}U,\ldots,\partial_{x_N}U\big\}.$$
Moreover, By Lemma C.2 of \cite{fls1}, for $j=1,\ldots,N$, $\partial_{x_j}U$ has the decay estimate
\begin{equation*}
|\partial_{x_j} U|\leq\frac{C }{1+|x|^{N+2s}}.
\end{equation*}

Throughout this paper, {\bf  we assume $$N\geq\max\{2s,2-2s\}$$} and set
\begin{align}\label{a*a0}
a_*=\int_{\R^N}U^2,\ \ \ a_0\equiv
\begin{cases}
+\infty,\,\, &\text{if}\,\,  p-1<\frac{4s}N,\vspace{0.12cm}\\
(ka_*)^{\frac{p-1}2},\,\,& \text{if} \,\, p-1=\frac{4s}N,\vspace{0.12cm}\\
 0,\,\,& \text{if}\, p-1>\frac{4s}N.
\end{cases}
\end{align}

Our first result is as follows.
\begin{Thm}\label{th1}
 Suppose that $u_a$ satisfies \eqref{eq1}-\eqref{eqnorm} concentrated at some points as $a\rightarrow a_0$. Then it holds that $\mu_a\rightarrow-\infty$ as $a\rightarrow a_0$ .

Moreover, if $p-1\leq\frac{4s}N$, there holds that
\begin{align}\label{equa}
u_a=\Big(\frac{-\mu_a}a\Big)^{\frac1{p-1}}\Big(\sum_{i=1}^kU\big((-\mu_a)^{\frac1{2s}}(x-x_{a,i})\big)+\omega_a\Big),
\end{align}
with $\displaystyle\int_{\R^N}\Big(-\frac1{\mu_a}|(-\Delta)^{\frac s2}\omega_a|^2+\omega_{a}^2\Big)=o\Big((-\mu_a)^{-\frac{N}{2s}}\Big)$.
\end{Thm}

We call $u_a$ a $k$-peak solution of \eqref{eq1}--\eqref{eqnorm} if $u_a$ satisfies \eqref{equa}. 
The problem has been extensively investigated if the critical points of $V(x)$ are isolated, while few is known otherwise.
In this paper, we assume that $V(x)$ obtains its local minimum or local maximum at a closed $N-1$ dimensional hyper-surface $\Gamma_i(i=1,\ldots,k)$, satisfying
$\Gamma_i\cap\Gamma_j=\emptyset$ if $i\neq j$. Precisely, we assume that

\vskip 0.2cm

{\bf (V)}. There exist some $\delta>0$ and some $C^2$ compact hyper surfaces  $\Gamma_i(i=1,\ldots,k)$ without boundary, satisfying
$$V(x)=V_i,\ \frac{\partial V(x)}{\partial \nu_i}=0,\ \frac{\partial^2 V(x)}{\partial \nu_i^2}\neq0,\ for\ any\ x\in\Gamma_i\ and\ i=1,\ldots,k,$$
where $V_i\in\R$, $\nu_i$ is the unit outward normal of $\Gamma_i$ at $x\in\Gamma_i$. Moreover, $V(x)\in C^4(\cup_{i=1}^kW_{\delta,i})$ 
with $W_{\delta,i}:=\{x\in\R^N:dist(x,\Gamma_i)<\delta\}.$

\begin{Rem}
We point out that the assumption (V) was first introduced in \cite{LPY19}, where they
studied the Bose-Einstein condensates. The assumption (V) implies that the potential $V(x)$ obtains its local minimum or local maximum on the hypersurface $\Gamma_i$ for
$i=1,\ldots,k$.
If $\delta>0$ is small, we set $\Gamma_{t,i}=\{x:V(x)=t\}\cap W_{\delta,i}$, which consists of two compact hypersurfaces in $\R^N$ without boundary for
$t\in[V_0,V_0+\sigma)$ or $t\in(V_0-\sigma,V_0]$ provided $\sigma>0$ is small enough. Moreover, the outward unit normal vector $\nu_{t,i}(x)$ and the $j-$th principal tangential unit vector $\tau_{t,i,j}(x),j=1,\ldots,N-1$ of $\Gamma_{t,i}$ at $x$ are Lip-continuous in $W_{\delta,i}$.
\end{Rem}

We can apply the Pohozaev identities to show that a $k$-peak solution of \eqref{eq1}-\eqref{eqnorm} must concentrate at some critical points of $V(x)$.
The following result explains where the concentrated points locate on $\Gamma=\cup_{i=1}^k\Gamma_i$.

\begin{Thm}\label{th2}
Under the condition (V), if $u_a$ is a $k$-peak solution of \eqref{eq1}-\eqref{eqnorm}, concentrating at $\{b_1,b_2,\ldots,b_k\}$ with $b_i\in\Gamma$,
 $b_i\neq b_j$ if $i\neq j$, and $|x_{a,i}-\Gamma_i|=O\big((-\mu_a)^{-\frac1{2s}}\big)$ as $a\rightarrow a_0$, then
\begin{align}\label{necess}
(D_{\tau_{i,j}}\Delta V)(b_i)=0,\ \ with\ i=1,\ldots,k\ and\ j=1,\ldots,N-1,
\end{align}
where $\tau_{i,j}$ is the $j-$th principal tangential unit vector of $\Gamma$ at $b_i$.
\end{Thm}
\medskip

To study the converse of  Theorem \ref{th2}, we need another non-degenerate condition on the critical point of $V(x)$. In fact, we define that $x_0\in\Gamma_i$ is non-degenerate
on $\Gamma_i$ if there holds that
$$\frac{\partial^2 V(x_0)}{\partial \nu_i^2}\neq0,\ and\ det\left(\Big(\frac{\partial^2\Delta V(x_0)}{\partial\tau_{i,l}\partial\tau_{i,j}}\Big)_{1\leq l,j\leq N-1}\right)\neq0.$$

\begin{Thm}\label{th3}
Assume that the condition (V) holds. If $b_i\in\Gamma$ are non-degenerate critical points of $V(x)$ on $\Gamma$ for $i=1,\ldots,k$ satisfying  \eqref{necess}
and $b_i\neq b_j$ for $i\neq j,$ then \eqref{eq1}--\eqref{eqnorm} has a $k$-peak solution $u_a$ concentrating at $b_1,\ldots,b_k$ as $a\rightarrow a_0$.

\end{Thm}
\smallskip

On the other hand,  if we assume the function $\Delta V(x)|_{x\in\Gamma_{i_0}}$ has
an isolated maximum point $b\in\Gamma_{i_0}$ with some $i_0\in\{1,\ldots,k\}$, that is  $\Delta V(x)<\Delta V(b)$
for all $x\in\Gamma_{i_0}\cap\big(B_\delta(b)\setminus b\big)$,
then we can
obtain a $k$-peak solution concentrating at one point.
\begin{Thm}\label{th3'}
Assume  \textup{($V$)} and $\frac{\partial^2V(x)}{\partial\nu_{i_0}^2}\neq0$ for any $x\in\Gamma_{i_0}$ with some $i_0\in\{1,\ldots,k\}$.
If $b\in \Gamma_{i_0}$ is an isolated maximum point of $ \Delta V(x)|_{x\in \Gamma_{i_0}}$ on $\Gamma_{i_0}$, then
  for any integer $k>0$,  problem \eqref{eq1}--\eqref{eqnorm} has a $k$-peak solution $u_a$ concentrating at $b$.
\end{Thm}

\medskip
Another main result of this paper is the following local uniqueness result.

\begin{Thm}\label{th4}

Suppose   \textup{($V$)}, and if further  $N\geq 2s+4$.~
 Let $u_a^{(1)}(x)$ and $u_a^{(2)}(x)$ be two  $k$-peak solutions of \eqref{eq1}--\eqref{eqnorm} concentrating  at $b_1,\cdots,b_k$ with $b_i\in \Gamma$,  and $b_i\ne b_j$ if  $i\ne j$.  If $ b_i$ is non-degenerate, $i=1,\cdots, k$,
 $
\displaystyle\sum^k_{i=1}\Delta V(b_i)\neq 0$ when $p-1=\frac{4s}N$,
and
\[
\Big(\frac{\partial^2 \Delta V(b_i)}{\partial \tau_{i,l} \partial\tau_{i,j}}\Big)_{1\leq l,j\leq N-1}+\frac{\partial  \Delta V(b_i)}{\partial \nu_i}  diag \big(\kappa_{i,1}, \cdots, \kappa_{i,N-1}\big),~\mbox{for}~i=1,\cdots,k
\]
is non-singular, where $\kappa_{i,j}$ is the $j$-th principal curvature of $\Gamma$ at $b_i$ for $j=1,\cdots,N-1$,
then there exists a small positive number $\sigma$,  such that $$u_a^{(1)}(x)\equiv u_a^{(2)}(x)$$ for all $a$ with $0<|a-(ka_*)^{\frac{p-1}2}|\le \sigma$ if  $p-1=\frac{4s}N$, or $0< a \le \sigma $ if $p-1>\frac{4s}N$, or $a\geq\frac1\sigma$ if $p-1<\frac{4s}N$.
\end{Thm}

\begin{Rem}
In fact, the result of Theorem \ref{th4} can hold for more general $p$ and $N.$ Here, in order to avoiding
writing too dispersively, we assume that $N\geq2s+4.$
\end{Rem}
As far as we know, there are very few results on the local uniqueness for the fractional Schr\"odinger equations, especially with the subcritical nonlinearities.
In particular, since the Lagrange multiplier $\mu_a$ in \eqref{eq1} depends on the solution $u_a$, the corresponding
linearized operator has changed, which brings  more nontrivial analysis, as mentioned in \cite{LPY19}.
For the
local uniqueness results of peak (or bubbling) solutions, the classical moving plane method is not available.
If $x_0$ is a non-degenerate critical point of $V(x)$, that is,  $ (D^2 V)$ is non-singular at $x_0$, one
can prove the local uniqueness of the peak solution concentrating at $x_0$ either by counting the local degree of the
corresponding reduced finite dimensional  problem as in  \cite{Cao3,CNY,G}, or by using Pohozaev type identities as in  \cite{Cao1,Deng,Grossi,GPY,GLW}.
One of the advantage in applying the Pohazaev identities can be found in dealing with the degenerate case(see \cite{Cao1,Deng,GPY}). In these results, the rate of degeneracy along each direction is the same,  although the critical point $x_0$ is degenerate. Following \cite{LPY19},
under the condition ($V$), the function $V(x)$ is non-degenerate along the normal direction $\nu_i$ of $\Gamma_i$. But along each tangential direction of $\Gamma_i$,
$V(x)$ is degenerate. Such non-uniform degeneracy makes the estimates more sophisticated.

\medskip

At the end of this section, we outline the main idea of the proof and discuss the main difficulties.
For the existence result, we first consider the following problem without constraint,
\begin{equation}\label{eq2}
\begin{cases}
(-\Delta)^s w+\big(\lambda+V(x)\big)w=w^{p},  &{x\in \R^N},\\[1mm]
w\in H^s(\R^N),
\end{cases}
\end{equation}
where $\lambda>0$ is a large parameter. For large $\lambda>0$,  by the standard reduction argument, we could construct various positive solutions concentrating at some stable points of $V(x)$. Particularly, we can construct positive $k-$peak solutions for \eqref{eq2} of the form
$$w_\lambda(x)=\lambda^{\frac1{p-1}}\Big(\sum_{j=1}^k U\big(\lambda^{\frac1{2s}}(x-x_{\lambda,i})\big)+\omega_\lambda\Big),$$
with $\displaystyle\int_{\R^N}\Big(\frac1\lambda|(-\Delta)^{\frac s2}\omega_\lambda|^2+\omega_\lambda^2\Big)=o\Big(\lambda^{-\frac{N}{2s}}\Big)$. 
Let $u_\lambda=\frac{w_\lambda}{\Big(\displaystyle\int_{\R^N}w_\lambda^2\Big)^{\frac12}}$. Then $\displaystyle\int_{\R^N}u_\lambda^2=1$, and
\begin{equation*}
\begin{cases}
(-\Delta)^s u_\lambda+\big(\lambda+V(x)\big)u_\lambda=a_\lambda u_\lambda^{p},  &{x\in \R^N},\\[1mm]
u_\lambda\in H^s(\R^N),
\end{cases}
\end{equation*}
with $$a_\lambda=\Big(\int_{\R^N}w_\lambda^2\Big)^{\frac{p-1}2}=\Big(k\lambda^{\frac2{p-1}-\frac N{2s}}\big(a_*+o(1)\big)\Big)^{\frac{p-1}2}.$$
We note that $a_\lambda>0$. Moreover, as $\lambda\rightarrow+\infty$, $a_\lambda\rightarrow a_0$, which is defined by \eqref{a*a0}.
Therefore, we obtain a concentrated solution with $k$-peaks for \eqref{eq2} with normalized $L^2$-norm, where $\mu=-\lambda$ and some suitable $a_\lambda$.
Hence, we are sufficed to answer a converse question that for any $a>0$  close to $a_0$,
whether one can choose a  suitable large $\lambda=\lambda_a>0$, such that \eqref{eq1}--\eqref{eqnorm} hold with
$\mu=-\lambda_a,\ \ u_a=\frac{w_{\lambda_a}}{\Big(\displaystyle\int_{\R^N}w_{\lambda_a}^2\Big)^{\frac12}},$
finally concluding Theorem \ref{th3}.
In fact, we solve the problem  by discussing
the relationship between $a$ and $\mu$.
\vskip 0.1cm

 Recently, for the classical BECs problem ($s=1,p=3,N=2,3$), a complete description of the solutions $u_a$ concentrating at $a_*$ related to problem was given in \cite{LPY19}. Specifically, in this case if $u_a$ is a solution of problem \eqref{eq1}--\eqref{eqnorm} with $s=1$, $p=3$ and $N=2,3$, which is concentrated at some points as $a\to a_0$ with $a_0\in \R$,  then it holds
\begin{equation*}
a_0 = k a_* > 0 \;\;\text{and}\;\; \mu=\mu_a\to -\infty~\mbox{as}~a\to a_0,
\end{equation*}
where $k$ is the number of peaks of the concentrated solution $u_a$. Moreover, the existence and local uniqueness of this kind concentrated solutions have been proved in \cite{LPY19}.
Unlike the argument used in \cite{LPY19}, a non-existence result, which is of great importance to prove the main result, could not be obtained  by a standard comparison argument in ODE theory.  We would apply a totally different method--the sliding method--corresponding to the non-local operator to show a counterpart result (Theorem \ref{th1}).

On the other hand,
we point out that, due to the non-local property of the operator $(-\Delta)^s$, we cannot construct the local Pohozaev identities directly, which, however,
is inevitable when we carry out the blow-up analysis especially in the proof of the local uniqueness.
To this end, we are supposed to study the corresponding harmonic extension problem (see section 4 and section 5), and have to
estimate several kind of integrals which never appear in the classic local Schr\"odinger problems. Similar arguments can be found in \cite{gnnt} etc..

Last but not least, throughout our discussion, we need
to distinct the different cases of $p-1<\frac{4s}N$, $p-1=\frac{4s}N$, and $p-1>\frac{4s}N$,
which are called respectively that the mass-subcritical, the mass-critical, and the mass-supercritical case.
Especially compared with the problem without any constraint,  in the proof of the local uniqueness result, we have to adopt one more different local Pohozaev identity,
associated with the mass-constraint condition.

\medskip

This paper is organized as follows.  In section~2, we  prove Theorem~\ref{th1} using
the sliding method.  Then in section~3, we estimate the Lagrange multiplier $\mu_a$
in terms of $a$.  The results for the location of the peaks and  for the existence of peak solutions are proved in  section~4, and the local uniqueness of
peak solutions are investigated in section~5.


\medskip

\section{ Proof of Theorem \ref{th1} by a non-existence result}
\begin{Lem}\label{lem2.1}
Assume that $P(x)$ satisfies $P(x)>1$ in $B_R(0)\setminus B_r(0)$ for some fixed $r>0$ and large $R>0$.
Then the problem $$(-\Delta)^su=P(x)u,\ \ u>0,\ \ in\ \R^N$$
has no solution.
\end{Lem}
\begin{proof}
We take $\lambda_1$ as the first eigenvalue of $(-\Delta)^s$
and $\phi$ as the associated eigenfunction satisfying
\begin{align}\begin{cases}
(-\Delta)^s\phi=\lambda_1\phi,& x\in\ B,\nonumber\\[1mm]
 \phi=0,&x\in \ B^c,
\end{cases}
\end{align}
where $B\subset B_R(0)\setminus B_r(0)$ is a large ball. Since $R>0$ is large enough, it may hold that $$\lambda_1\leq1<P(x)~~\mbox{in}~~B.$$

Now, we set $w(x):=\displaystyle\max_{x\in B}\frac{\phi(x)}{u(x)}\cdot u(x)$. Obviously, $w(x)\geq\phi(x)$  and $w(x_0)=\phi(x_0)$ with $x_0\in B$ the maximal point of $\frac{\phi(x)}{u(x)}$.  To get a contradiction, we are sufficed to prove that $w(x)>\phi(x)$ by use of the maximum  principle.

In fact, if we denote $\gamma=\displaystyle\max_{x\in B}\frac{\phi(x)}{u(x)}$, then there holds that
\begin{align*}
(-\Delta)^sw=\gamma(-\Delta)^su=\gamma P(x)u\geq \gamma\lambda_1u=\lambda_1\max_{x\in B}\frac{\phi(x)}{u(x)}u\geq\lambda_1\phi=(-\Delta)^s\phi,
\end{align*}
which gives that $$(-\Delta)^sw\geq(-\Delta)^s\phi~\,~\mbox{in}~~B.$$ On the other hand, $w\geq\phi$ in $B^c$ obviously. Moreover, since here $w\not\equiv\phi$, the strong maximum principle implies that
$w>\phi$ concluding the Lemma by contradiction.
\end{proof}

\begin{Rem}
For the classical BECs problem when $s=1,p=3$ and $N=2,3$ in \cite{LPY19}, the standard comparison method from the ODE theory is sufficed to show such a non-existence
result, which, however,  is invalid in the fractional case. We take the sliding method to overcome this obstacle, which, developed by Chen, Li and Zhu \cite{chenlizhu},  would be typically available in the study of some other similar problems when dealing with the fractional operator.

\end{Rem}
\smallskip

\begin{proof}[\textbf{Proof of Theorem \ref{th1}}]~We divide the proofs into following three steps.
\vskip 0.2cm

\noindent Step 1.  We first show $\mu_a\rightarrow-\infty$. By contradiction, we first suppose $|\mu_a|\leq M$. Since $\displaystyle\int_{\R^N}u_a^2=1$,
 we could use the argument in \cite{jlx}
to get $u_a$ is uniformly bounded, which means
that $u_a$ does not blow up. Then if $\mu_a\rightarrow+\infty$, we set $P(x)=\mu_a-V(x)+au_a^{p-1}$. Since $u_a$ concentrates at some
points, we may assume $$au_a^{p-1}>-1~~\mbox{in}~~\R^N\setminus B_r(0)~~\mbox{for some}~~r>0.$$ Hence, for any $R>0$, we have $$P(x)>1~~\mbox{for}~~x\in B_R(0)\setminus B_r(0),$$ which gives a contradiction by Lemma \ref{lem2.1}.

\vskip 0.2cm

\noindent Step 2. It holds that $a>0$. If not, we could obtain then $$(-\Delta)^su_a\leq0.$$ By
the maximum principle corresponding to the fractional operator \cite{jlx}, we find $u_a\leq 0$, which is a contradiction.

\vskip 0.2cm

\noindent Step 3. In the case of $p-1\leq \frac{4s}{N}$. Denote $v_a(x)=(-\mu_a)^{-\frac 1{p-1}}u_a\big((-\mu_a)^{-\frac 1{2s}}x\big)$. Then
\begin{align*}
(-\Delta)^sv_a+\Big(1+\frac1{-\mu_a}V\big((-\mu_a)^{-\frac 1{2s}}x\big)\Big)v_a=av_a^p
\end{align*}
and $\displaystyle\int_{\R^N}v_a^2=(-\mu_a)^{\frac N{2s}-\frac2{p-1}}$.
Therefore, we apply the Moser iteration to get $v_a$ is bounded uniformly in $a$.
Let $y_a$ be a maximum point of $v_a$. Then we could show that $a$ is bounded from below. Applying the   blow-up argument established by \cite{jlx}, we obtain that
there exists some integer $k>0$ such that
$$
v_a=\sum_{i=1}^kU_{a_{0}}(x-y_{a,i})+\bar \omega_a,
$$
for some $y_{a,i}\in\R^N$ with $|y_{a,i}-y_{a,j}|\rightarrow+\infty$ for $i\neq j$, where $U_{a_{0}}$ is the unique positive solution of
\begin{equation*}
\begin{cases}
(-\Delta)^sU_{a_{0}}+U_{a_{0}}=a_{0}U_{a_{0}}^p,\\[2mm]
 U_{a_{0}}(0)=\displaystyle\max_{x\in\R^N}U_{a_{0}}(x),
 \end{cases}
\end{equation*}
and
 $$\displaystyle\int_{\R^N}\big(|(-\Delta)^{\frac s2}\bar \omega_a|^2+\bar \omega_a^2\big)=o(1)$$
Since $\displaystyle\int_{\R^N}u_a^2=1$,  we have $$\int_{\R^N}v_a^2=(-\mu_a)^{\frac N{2s}-\frac2{p-1}}.$$ In view of $U_a=a_{0}^{-\frac1{p-1}}U$, there holds that
$$(-\mu_a)^{\frac N{2s}-\frac2{p-1}}=ka_0^{-\frac 2{p-1}}a_*+o(1),$$ implying the conclusion.

\end{proof}

\section{Some estimates by Pohozaev identities}

We use Pohozaev identities  to estimate $\mu_a$ with respect to $a$.
Let $\varepsilon=(-\mu_a)^{-\frac1{2s}}$ and $v_a=(-\frac{\mu_a}{a})^{-\frac1{p-1}}u_a$.
Then \eqref{eq1} turns to
\begin{align}\label{eqv}
\varepsilon^{2s}(-\Delta)^sv_a+\big(1+\varepsilon^{2s}V(x)\big)v_a=v_a^p,\,\,\,\,v_{a}\in H^{s}(\R^{N}).
\end{align}
For any $a>0$, we define the norm $\|v\|_a=\Big(\displaystyle\int_{\R^N}\varepsilon^{2s}|(-\Delta)^{\frac s2}v|^2+v^2\Big)^{\frac12}$.

Hence, a $k$-peak solution of \eqref{eqv} has the form of
$$v_a(x)=\sum_{i=1}^kU_{\varepsilon,x_{a,i}}(x)+\varphi_a(x),$$
where $$U_{\varepsilon,x_{a,i}}(x)=(1+\varepsilon^{2s}V_i)^{\frac1{p-1}}
U\Big(\frac{(1+\varepsilon^{2s}V_i)^{\frac1{2s}}(x-x_{a,i})}{\varepsilon}\Big),\ \
\|\varphi_a\|_a=o(\varepsilon^{\frac N2}).$$
We write then
\begin{align*}
\varepsilon^{2s}(-\Delta)^s\varphi_a+\Big(\big(1+\varepsilon^{2s}V(x)\big)-p\sum_{i=1}^kU_{\varepsilon,x_{a,i}}^{p-1}\Big)\varphi_a=N_a(\varphi_a)+l_a(x),
\end{align*}
where
\begin{align*}
l_a(x)=-\varepsilon^{2s}\sum_{i=1}^k\big(V(x)-V_i\big)U_{\varepsilon,x_{a,i}}
+\Big(\sum_{i=1}^kU_{\varepsilon,x_{a,i}}\Big)^p-\sum_{i=1}^kU_{\varepsilon,x_{a,i}}^p,
\end{align*}
and
\begin{align}\label{Na}
N_a(\varphi_a)=\Big(\sum_{i=1}^kU_{\varepsilon,x_{a,i}}+\varphi_a\Big)^p-\Big(\sum_{i=1}^kU_{\varepsilon,x_{a,i}}\Big)^p-p\sum_{i=1}^kU_{\varepsilon,x_{a,i}}^{p-1}\varphi_a.
\end{align}
Set
\begin{align*}
E_{a,x_{a,i}}:=\Big\{u\in H^s(\R^N):\langle u,\frac{\partial U_{\varepsilon,x_{a,i}}}{\partial x_j}\rangle_a=0,j=1,\ldots, N\Big\}.
\end{align*}
One may move $x_{a,i}$ such that the perturbing term $\varphi_a\in\displaystyle\bigcap_{i=1}^k E_{a,x_{a,i}}$.

\vskip 0.1cm

Let $L_a$ be the bounded linear operator from $H^s(\R^N)$ to itself, with the form of
\begin{align}\label{La}
\langle L_a u,v\rangle_a=\int_{\R^N}\varepsilon^{2s}(-\Delta)^{\frac s2}u(-\Delta)^{\frac s2}v+\big(1+\varepsilon^{2s}V(x)\big)uv-p\sum_{i=1}^kU_{\varepsilon,x_{a,i}}^{p-1}uv.
\end{align}
It is standard to prove the following result.
\begin{Lem}
There exists some constant $c>0$, such that for all $a$ around $a_0$, it holds then
$$\|L_au\|_a\geq c\|u\|_a,\ \ for\ all\ u\in\bigcap_{i=1}^k E_{a,x_{a,i}}.$$

\end{Lem}

\begin{Lem}\label{lemla}
A $k$-peak solution $v_a$ of \eqref{eqv} concentrating at $b_1,\ldots,b_k$ has the following form
\begin{align*}
v_a(x)=\sum_{i=1}^kU_{\varepsilon,x_{a,i}}(x)+\varphi_a(x),
\end{align*}
with $\varphi_a\in \bigcap_{i=1}^k E_{a,x_{a,i}}$ and
\begin{align}\label{phia}
\|\varphi_a\|_a=&O\Big(|\sum_{i=1}^k(V(x_{a,i})-V_i)|\varepsilon^{\frac N2+2s}+|\sum_{i=1}^k\nabla V(x_{a,i})|\varepsilon^{\frac N2+2s+1}+\varepsilon^{\frac N2+2s+2}\Big)\nonumber\\
&+\begin{cases}O\Big(\varepsilon^{\frac N2+\frac p2(N+2s)}\Big), \,&\text{if} \,\,1<p\leq2,\\[2mm] O\Big(\varepsilon^{\frac N2+ (N+2s)}\Big),\,&\text{if}\,\, p>2.
\end{cases}
\end{align}

\end{Lem}

\begin{proof}
By calculation, we have for any $v\in H^s(\R^N)$,
\begin{align*}
\langle l_a(x),v\rangle_a =&-\varepsilon^{2s}\sum_{i=1}^k\int_{\R^N}
\big(V(x)-V_i\big)U_{\varepsilon,x_{a,i}}v+\int_{\R^N}\big(
(\sum_{i=1}^kU_{\varepsilon,x_{a,i}})^p-\sum_{i=1}^kU_{\varepsilon,x_{a,i}}^p\big)v\\
 =&-\varepsilon^{2s}\sum_{i=1}^k\int_{\R^N}
\big(V(x)-V(x_{a,i})\big)U_{\varepsilon,x_{a,i}}v-\varepsilon^{2s}\sum_{i=1}^k\int_{\R^N}
\big(V(x_{a,i})-V_i\big)U_{\varepsilon,x_{a,i}}v\\
&+ \int_{\R^N}\big((\sum_{i=1}^kU_{\varepsilon,x_{a,i}})^p-\sum_{i=1}^kU_{\varepsilon,x_{a,i}}^p\big)v:=-A_1-A_2+A_3.
\end{align*}
We estimate $A_1$ as follows
\begin{align*}
A_1=& \varepsilon^{2s}\sum_{i=1}^k\int_{\R^N}
\big(V(x)-V(x_{a,i})\big)U_{\varepsilon,x_{a,i}}v\\
\leq& \varepsilon^{2s}\sum_{i=1}^k\int_{B_{\delta}(x_{a,i})}
\big(\nabla V(x_{a,i})\cdot(x-x_{a,i})+O(|x-x_{a,i}|^2)\big)U_{\varepsilon,x_{a,i}}v\\&+
C\varepsilon^{2s}\Big(\sum_{i=1}^k\int_{B^c_{\delta}(x_{a,i})}U^2_{\varepsilon,x_{a,i}}\Big)^{\frac12}\|v\|_{L^2}\\
\leq &C\varepsilon^{2s+\frac{N}{2}+1}|\sum_{i=1}^k\nabla V(x_{a,i})|\cdot\|v\|_{a}+\varepsilon^{2s+\frac{N}{2}+2}\|v\|_{a}.
\end{align*}
Similarly,
\begin{align*}
A_2=& \varepsilon^{2s}\sum_{i=1}^k\int_{\R^N}
\big(V(x_{a,i})-V_i\big)U_{\varepsilon,x_{a,i}}v\leq C\varepsilon^{2s+\frac{N}{2}}\sum_{i=1}^k|V(x_{a,i})-V_i|\cdot\|v\|_{a}.
\end{align*}
Finally, since
\begin{align*}
|A_3|=&\Big|\int_{\R^N}\big((\sum_{i=1}^kU_{\varepsilon,x_{a,i}})^p-\sum_{i=1}^kU_{\varepsilon,x_{a,i}}^p\big)v\Big|
=\begin{cases}O\Big(\displaystyle\int_{\R^N}\sum_{i\neq j}U^{p-1}_{\varepsilon,x_{a,i}}U_{\varepsilon,x_{a,j}}v\Big),&\text{if}\,\,p>2,\vspace{0.12cm}\\
O\Big(\displaystyle\int_{\R^N}\sum_{i\neq j}U^{\frac p2}_{\varepsilon,x_{a,i}}U^{\frac p2}_{\varepsilon,x_{a,j}}v\Big),&\text{if}\,\,1<p\leq2,
\end{cases}
\end{align*}
then by Lemma \ref{lemA1}, we compute directly that if $p>2$
\begin{align*}
&\Big|\int_{\R^N}\sum_{i\neq j}U^{p-1}_{\varepsilon,x_{a,i}}U_{\varepsilon,x_{a,j}}v\Big|
\leq C\Big(\int_{\R^N}\sum_{i\neq j}U^{2(p-1)}_{\varepsilon,x_{a,i}}U^2_{\varepsilon,x_{a,j}}\Big)^{\frac12}\|v\|_{a}\\
&\leq C\varepsilon^{\frac N2}\sum_{i\neq j}\Big(\int_{\R^N}\frac{1}{(1+|y-\frac{x_{a,j}}{\varepsilon}|)^{2(N+2s)(p-1)}} \frac{1}{(1+|y-\frac{x_{a,i}}{\varepsilon}|)^{2(N+2s)}}\Big)^{\frac12}\|v\|_{a}\\
&\leq C\varepsilon^{\frac N2}\sum_{i\neq j}\frac{1}{| \frac{x_{a,j}-x_{a,i}}{\varepsilon}|^{N+2s}}\Big(\int_{\R^N}\frac{1}{(1+|y-\frac{x_{a,j}}{\varepsilon}|)^{2(N+2s)(p-1)}}+ \frac{1}{(1+| y-\frac{x_{a,i}}{\varepsilon}|)^{2(N+2s)(p-1)}}\Big)^{\frac12}\|v\|_{a}\\
&\leq C\varepsilon^{\frac N2+N+2s}\|v\|_{a};
\end{align*}
while for $1<p\leq 2$,
\begin{align*}
&\Big|\int_{\R^N}\sum_{i\neq j}U^{\frac p2}_{\varepsilon,x_{a,i}}U^{\frac p2}_{\varepsilon,x_{a,j}}v\Big|
\leq C\Big(\int_{\R^N}\sum_{i\neq j}U^{p}_{\varepsilon,x_{a,i}}U^p_{\varepsilon,x_{a,j}}\Big)^{\frac12}\|v\|_{a}\\
&\leq C\varepsilon^{\frac N2}\sum_{i\neq j}\Big(\int_{\R^N}\frac{1}{(1+|y-\frac{x_{a,j}}{\varepsilon}|)^{p(N+2s) }} \frac{1}{(1+|y-\frac{x_{a,i}}{\varepsilon}|)^{ p(N+2s)}}\Big)^{\frac12}\|v\|_{a}\\
&\leq C\varepsilon^{\frac N2}\sum_{i\neq j}\frac{1}{| \frac{x_{a,j}-x_{a,i}}{\varepsilon}|^{\frac p2(N+2s)}}\|v\|_{a}\leq C\varepsilon^{\frac N2+\frac p2(N+2s)}\|v\|_{a}.
\end{align*}

To sum up, the result has been proved.

\end{proof}

\begin{Lem}
It holds that
\begin{align*}
\int_{\R^N}U^2=\Big(\frac N{s(p+1)}-\frac N{2s}+1\Big)\int_{\R^N}U^{p+1}.
\end{align*}
\end{Lem}
\begin{proof}
It is directly obtained by the following Pohozaev identities:
\begin{align*}
 \frac{N-2s}2\int_{\R_+^{N+1}}t^{1-2s}|\nabla\tilde U|^2=-\frac N2\int_{\R^N}U^2+\frac N{p+1}\int_{\R^N}U^{p+1}
\end{align*}
and
\begin{align*}\int_{\R_+^{N+1}}t^{1-2s}|\nabla\tilde U|^2=- \int_{\R^N}U^2+ \int_{\R^N}U^{p+1}.
\end{align*}
\end{proof}

\begin{Prop}\label{prop3.4}
There holds that
\begin{align*}
& a^{\frac2{p-1}}=\sum_{i=1}^k(V_i-\mu)^{\frac2{p-1}-\frac N{2s}}\int_{\R^N}U^2+
O\Big(|\sum_{i=1}^k(V(x_{a,i})-V_i)|(-\mu)^{\frac2{p-1}-1-\frac N{2s}}\\&\quad\quad\quad+|\sum_{i=1}^k\nabla V(x_{a,i})|(-\mu)^{\frac2{p-1}-1-\frac N{2s}-\frac1{2s}}+(-\mu)^{\frac2{p-1}-1-\frac N{2s}-\frac1s}\Big)\\
&\quad\quad\quad
+O\Big(\sum_{i\neq j}(-\mu)^{\frac{2}{p-1}-\frac{N}{s}-1}|x_{a,j}-x_{a,i}|^{-(N+2s)}\Big)
+\begin{cases}O\Big((-\mu)^{\frac2{p-1}-\frac N{2s}-\frac{pN}{4s}-\frac p2}\Big), &\text{if}\,\, 1<p\leq2,\\[2mm] O\Big((-\mu)^{\frac2{p-1}-\frac N{s}-1}\Big),&\text{if}\,\, p>2.
\end{cases}
\end{align*}
\end{Prop}

\begin{proof}
Let $u_a$ be a solution of \eqref{eq1}-\eqref{eqnorm}.                                
It holds that
\begin{align*}
1&=\int_{\R^N}|u_a(x)|^2=\int_{\R^N}\left(\sum_{i=1}^k(\frac{V_i-\mu}{a})^{\frac1{p-1}}U((V_i-\mu)^{\frac1{2s}}(x-x_{a,i}))+(\frac{-\mu}{a})^{\frac1{p-1}}\varphi_a
\right)^2\\
&=\sum_{i=1}^k\frac{(V_i-\mu)^{\frac2{p-1}-\frac N{2s}}}{a^{\frac2{p-1}}}\int_{\R^N}U^2
+\sum_{i=1}^k\frac{(-\mu)^{\frac1{p-1}}}{a^{\frac2{p-1}}}(V_{i}-\mu)^{\frac{1}{p-1}-\frac{N}{2s}}
\int_{\R^N}U(x)\varphi_a((V_{i}-\mu)^{-\frac{1}{2s}}x+x_{a,i})
\\
&\quad+\sum_{i\neq j}\frac{(V_{i}-\mu)^{\frac{1}{p-1}-\frac{N}{2s}}}{a^{\frac{2}{p-1}}}(V_{j}-\mu)^{\frac{1}{p-1}}
\int_{\R^N}U(x)U\Big((\frac{V_{j}-\mu}{V_{i}-\mu})^{\frac{1}{2s}}x-(V_{j}-\mu)^{\frac{1}{2s}}(x_{a,j}-x_{a,i})\Big)
\\
&\quad+\frac{(-\mu)^{\frac2{p-1}}}{a^{\frac2{p-1}}}\int_{\R^N}\varphi_a^2\\
&=\sum_{i=1}^k\frac{(V_i-\mu)^{\frac2{p-1}-\frac N{2s}}}{a^{\frac2{p-1}}}\int_{\R^N}U^2+
\sum_{i=1}^k\frac{(-\mu)^{\frac1{p-1}}}{a^{\frac2{p-1}}}(V_{i}-\mu)^{\frac{1}{p-1}-\frac{N}{4s}}O(\|\varphi_a\|_{L^{2}})
\\
&\quad+O\Big(\sum_{i\neq j}
\frac{(V_{i}-\mu)^{\frac{1}{p-1}-\frac{N}{2s}}(V_{j}-\mu)^{\frac{1}{p-1}-\frac{N}{2s}-1}}{a^{\frac{2}{p-1}}}|x_{a,j}-x_{a,i}|^{-(N+2s)}\Big)
+\frac{(-\mu)^{\frac2{p-1}}}{a^{\frac2{p-1}}}\int_{\R^N}\varphi_a^2\\
&=\sum_{i=1}^k\frac{(V_i-\mu)^{\frac2{p-1}-\frac N{2s}}}{a^{\frac2{p-1}}}\int_{\R^N}U^2+
\frac{(-\mu)^{\frac2{p-1}-\frac N{4s}}}{a^{\frac2{p-1}}}O\Big(|\sum_{i=1}^k(V(x_{a,i})-V_i))|\varepsilon^{\frac N2+2s}\\
&\quad+|\sum_{i=1}^k\nabla V(x_{a,i})|\varepsilon^{\frac N2+2s+1}+\varepsilon^{\frac N2+2s+2}\Big)
+\frac{(-\mu)^{\frac2{p-1}-\frac N{4s}}}{a^{\frac2{p-1}}}
\begin{cases}O\Big(\varepsilon^{\frac N2+\frac p2(N+2s)}\Big), \,\,&\text{if}\,\, 1<p\leq2,\\[3mm]
O\Big(\varepsilon^{\frac N2+ (N+2s)}\Big),\,\,&\text{if}\,\, p>2,\end{cases}\\
&\quad+O\Big(\sum_{i\neq j}
\frac{(-\mu)^{\frac{2}{p-1}-\frac{N}{s}-1}}{a^{\frac{2}{p-1}}}|x_{a,j}-x_{a,i}|^{-(N+2s)}\Big),
\end{align*}
which gives the result.
\end{proof}

\section{Existence of the Peak Solutions}

\subsection{Locating of the peaks}$\,$  \vskip 0.2cm

Let $v_a$  be a peak solution of \eqref{eqv}.
Following the argument in section 5 of \cite{davilapinowei}, we obtain that the concentrating solution $v_a$ to \eqref{eqv} must satisfy that
\begin{align}\label{decay}
\sup_{x\in\R^N}\Big(\sum_{i=1}^k\frac{1}{\big(1+|\frac{x-x_{a,i}}{\varepsilon}|^2\big)^{\frac{N+2s-\theta}{2}}}\Big)^{-1}|v_a(x)|<+\infty,
\end{align}
where $\theta$ is any constant such that $\theta\in(0,\frac N2+2s)$.
Now we quote the extension of $v_a$ and its equation:
\begin{equation}\label{eqexv}
\begin{cases}
div(t^{1-2s}\nabla\tilde v_a)=0,  &{x\in\R^{N+1}_+},\vspace{0.2cm}\\
-\varepsilon^{2s}\displaystyle\lim_{t\rightarrow0}t^{1-2s}\partial_t\tilde v_a(y,t)= v_a^p-\big(1+\varepsilon^{2s}V(x)\big)v_a,   &{x\in\R^{N}}.
\end{cases}
\end{equation}

Denote
\begin{align*}
&B_\rho(x_0)=\big\{y\in\mathbb R^N:|y-x_0|\leq\rho\big\}\subseteq\mathbb R^{N},\\
&\mathbf{B}_\rho^+(x_0)=\big\{Y=(y,t):|Y-(x_0,0)|\leq\rho,t>0\big\}\subseteq\mathbb R_+^{N+1},\\
&\partial'\mathbf{B}_\rho^+(x_0)=\big\{Y=(y,t):|Y-(x_0,0)|\leq\rho,t=0\big\}\subseteq\mathbb R^{N},\\
&\partial''\mathbf{B}_\rho^+(x_0)=\big\{Y=(y,t):|Y-(x_0,0)|=\rho,t>0\big\}\subseteq\mathbb R_+^{N+1},\\
&\partial\mathbf{B}_\rho^+(x_0)=\partial'\mathbf{B}_\rho^+(x_0)\cup\partial''\mathbf{B}_\rho^+(x_0).
\end{align*}
Multiplying \eqref{eqexv} by $\frac{\partial\tilde v_a}{\partial x_j}$ and integrating by parts on $B_\rho(x_{a,i})$(or $\mathbf B_\rho^+(x_{a,i})$) with some $\rho>0$, we have the following Pohozaev identities:
\begin{equation}\label{poh1}
\begin{split}
 \varepsilon^{2s}\int_{B_\rho(x_{a,i})}\frac{\partial V(x)}{\partial x_j}v_a^2
&=-2\varepsilon^{2s}\int_{\partial''\mathbf B_\rho^+(x_{a,i})}t^{1-2s}\frac{\partial\tilde v_a}{\partial\nu}\frac{\partial\tilde v_a}{\partial y^j}
+\varepsilon^{2s}\int_{\partial''\mathbf B_\rho^+(x_{a,i})}t^{1-2s}|\nabla\tilde v_a|^2\nu_j \\&
\quad+ \int_{\partial B_\rho(x_{a,i})} \big(1+\varepsilon^{2s}V(x)\big)v_a^2\nu_j
-\frac2{p+1} \int_{\partial B_\rho(x_{a,i})} v_a^{p+1}\nu_j.
\end{split}\end{equation}
Using Lemma \ref{lemA3} and Remark \ref{remA2}, we estimate that for $(y-x)^2+t^2=\rho^2,t>0$, $\theta\in(0,\frac N2+2s)$,
\begin{align*}
|\tilde v_a(y,t)|
& \leq C\sum_{i=1}^k\int_{\R^N}\frac{t^{2s}}{(|y-\xi|+t)^{N+2s}}\frac1{(1+|\frac{\xi-x_{a,i}}{\varepsilon}|)^{N+2s-\theta}}d\xi\\
&\leq C\begin{cases}
\varepsilon^{N+2s-\theta}\displaystyle\sum_{i=1}^k\frac1{(1+|y-x_{a,i}|)^{N+2s-\theta}},\ &\theta> 2s,\\
\varepsilon^{N}\displaystyle\sum_{i=1}^k\frac1{(1+|y-x_{a,i}|)^{N+2s-\theta}},\ &\theta<2s.
\end{cases}
\end{align*}
Similarly, we could also prove that
\begin{align*}
\Big|\frac{\partial}{\partial y^j}\tilde v_a(y,t)\Big|,\,\,\,\Big|\frac{\partial}{\partial \nu}\tilde v_a(y,t)\Big|
& \leq C\begin{cases}
\varepsilon^{N+2s-\theta}\displaystyle\sum_{i=1}^k\frac1{(1+|y-x_{a,i}|)^{N+2s-\theta}},\ &\theta> 2s,\\
\varepsilon^{N}\displaystyle\sum_{i=1}^k\frac1{(1+|y-x_{a,i}|)^{N+2s-\theta}},\ &\theta< 2s.
\end{cases}
\end{align*}

Therefore, we estimate the boundary terms as follows: if $\theta>2s$
\begin{align*}
 &-2\varepsilon^{2s}\int_{\partial''\mathbf B_\rho^+(x_{a,i})}t^{1-2s}\frac{\partial\tilde v_a}{\partial\nu}\frac{\partial\tilde v_a}{\partial y^j}
\nonumber\\
&\leq C\varepsilon^{2s}\varepsilon^{2(N+2s-\theta)}\int_{\partial''\mathbf B_\rho^+(x_{a,i})}t^{1-2s}\Big(\sum_{i=1}^k\frac1{(1+|y-x_{a,i}|)^{N+2s-\theta}}\Big)^2\leq C\varepsilon^{2(N+3s-\theta)};
\end{align*}
and similarly,
\begin{align*}
\varepsilon^{2s}\int_{\partial''\mathbf B_\rho^+(x_{a,i})}t^{1-2s}|\nabla\tilde v_a|^2\nu_j\leq C\varepsilon^{2(N+3s-\theta)};
\end{align*}
\begin{align*}
\int_{\partial B_\rho(x_{a,i})} \big(1+\varepsilon^{2s}V(x)\big)v_a^2\nu_j\leq C\varepsilon^{2(N+2s-\theta)};
\end{align*}

\begin{align*}
\frac2{p+1} \int_{\partial B_\rho(x_{a,i})} v_a^{p+1}\nu_j\leq C\varepsilon^{(p+1)(N+2s-\theta)}.
\end{align*}
While if $\theta<2s$,
\begin{align*}
 -2\varepsilon^{2s}\int_{\partial''\mathbf B_\rho^+(x_{a,i})}t^{1-2s}\frac{\partial\tilde v_a}{\partial\nu}\frac{\partial\tilde v_a}{\partial y^j}
&\leq C\varepsilon^{2s}\varepsilon^{2N}\int_{\partial''\mathbf B_\rho^+(x_{a,i})}t^{1-2s}\Big(\sum_{i=1}^k\frac1{(1+|y-x_{a,i}|)^{N+2s-\theta}}\Big)^2\nonumber\\
&\leq C\varepsilon^{2(N+s)};
\end{align*}
\begin{align*}
&\varepsilon^{2s}\int_{\partial''\mathbf B_\rho^+(x_{a,i})}t^{1-2s}|\nabla\tilde v_a|^2\nu_j\leq C\varepsilon^{2(N+s)};\end{align*}
\begin{align*}
&\int_{\partial B_\rho(x_{a,i})} \big(1+\varepsilon^{2s}V(x)\big)v_a^2\nu_j\leq C\varepsilon^{2(N+2s-\theta)};\end{align*}
and
\begin{align*} \int_{\partial B_\rho(x_{a,i})} v_a^{p+1}\nu_j\leq C\varepsilon^{(p+1)(N+2s-\theta)}.
\end{align*}
Hence \eqref{poh1} gives then for $\theta\in(0,\frac N2+2s)$,
\begin{align}\label{poh2'}
\int_{B_\rho(x_{a,i})}\frac{\partial V(x)}{\partial x_j}v_a^2
=\begin{cases}O\big(\varepsilon^{2(N+s-\theta)}\big),\ &\theta>2s,\\[2mm]
O\big(\varepsilon^{\min\{2N,2(N+s-\theta)\}}\big),\ &\theta<2s,
\end{cases}
\end{align}
which implies the first necessary condition for the concentrated points $b_i$:
$$\nabla V(b_i)=0,\ \ for\ i=1,\ldots, k.$$
\medskip

\begin{Rem}
In our discussion,  since we are sufficed to consider certain small $\theta\in(0,\frac N2+2s)$ fixed, then for simplicity,  {\bf we may as well suppose $$\theta<s.$$}
Hence \eqref{poh2'} turns to be
\begin{align}\label{poh2}
\int_{B_\rho(x_{a,i})}\frac{\partial V(x)}{\partial x_j}v_a^2
=O(\varepsilon^{2N}).
\end{align}

\end{Rem}

\medskip
\begin{proof}[\textbf{Proof of Theorem \ref{th2}}]

In view of $x_{a,i}\rightarrow b_i\in\Gamma_i$,  if $\Gamma_i$ is a local minimum set of $V(x)$, there exists some $t_a\in[V_i,V_i+\theta]$;
while if $\Gamma_i$ is a local maximum set of $V(x)$, there exists some $t_a\in[V_i-\theta,V_i]$, such that $x_{a,i}\in\Gamma_{t_a,i}$.
\vskip 0.1cm

Let $H(x)=\langle\nabla V(x),\tau_{a,i}\rangle$ and $\tau_{a,i}$ be the unit tangential vector of $\Gamma_{t_a,i}$ at $x_{a,i}$. Then $$H(x_{a,i})=0.$$
On the one hand, we have the expansion for $x\in B_\rho(x_{a,i})$,
\begin{align*}
H(x)=\langle\nabla H(x_{a,i}),x-x_{a,i}\rangle+\frac12 \big\langle  \langle\nabla^2 H(x_{a,i}),x-x_{a,i}\rangle,x-x_{a,i}\big\rangle+o\big(|x-x_{a,i}|^2\big),
\end{align*}
and then,
\begin{equation}\label{5}
\begin{split}
 &\int_{B_\rho(x_{a,i})}H(x)U_{\varepsilon,x_{a,i}}^2(x) \\
&=-2\int_{B_\rho(x_{a,i})}H(x)U_{\varepsilon,x_{a,i}}(x)\varphi_a-\int_{B_\rho(x_{a,i})}H(x) \varphi_a^2+O(\varepsilon^{2(N+s-\theta)}) \\&
=O\left((\varepsilon^{\frac N2+1}|\nabla H(x_{a,i})|+\varepsilon^{\frac N2+2})\|\varphi_a\|_a+\varepsilon|\nabla H(x_{a,i})|\cdot \|\varphi_a\|_a^2\right)+O\Big(\varepsilon^{2(N+s-\theta)}\Big) \\
&=O\Big(\varepsilon^{N+2s+2}\Big),
\end{split}
\end{equation}
where we used the fact that, by assumption,  $$|V(x_{a,i})-V_i|=O(\varepsilon)~~\mbox{and}~~\|\varphi_a\|_a=O\big(\varepsilon^{\frac N2+2s+1}\big).$$

On the other hand,  since $H(x_{a,i})=0$,
\begin{align*}
&\int_{B_\rho(x_{a,i})}H(x)U_{\varepsilon,x_{a,i}}^2(x)
=\frac1{2N}\varepsilon^{N+2}\Delta H(x_{a,i})\int_{\R^N}|x|^2U^2+O(\varepsilon^{N+4})
\end{align*}
which, combined with \eqref{5}, gives that $$\Delta H(x_{a,i})=O(\varepsilon^{2s}).$$
Hence  from the condition {\bf (V)}, we get \eqref{necess} concluding Theorem \ref{th2}.
\end{proof}

\smallskip
\subsection{Existence of the normalized peak solutions}$\,$  \vskip 0.2cm

For the existence result of problem \eqref{eq2} with large $\lambda>0$, we set $\eta=\lambda^{-\frac1{2s}}$ and $w(x)\mapsto \lambda^{\frac1{p-1}}w(x)$, and
\eqref{eq2} turns to
\begin{equation}\label{eq2'}
\begin{cases}
\eta^{2s}(-\Delta)^s w+\big(1+\eta^{2s}V(x)\big)w=w^{p},  &{x\in \R^N},\\[2mm]
w\in H^s(\R^N),
\end{cases}
\end{equation}
which is equivalent to
\begin{equation*}
\begin{cases}
div(t^{1-2s}\nabla\tilde w)=0,  &{x\in\R^{N+1}_+},\\[1mm]
-\eta^{2s}\displaystyle\lim_{t\rightarrow0}t^{1-2s}\partial_t\tilde w(y,t)= w^p-\big(1+\eta^{2s}V(x)\big)w,   &{x\in\R^{N}}.
\end{cases}
\end{equation*}
We denote $\langle u,v\rangle_\eta=\int_{\R^N}\big(\eta^{2s}(-\Delta)^{\frac s2}u(-\Delta)^{\frac s2}v+uv\big)$, and $\|u\|_\eta=\langle u,u\rangle_\eta^{\frac12}$.
In the following, for $\eta>0$ small, we construct a $k$-peak solution $u_\eta$ to \eqref{eq2'} concentrating at $b_1,\ldots,b_k$.

\begin{Prop}\label{propeta}
There exists an $\eta_0>0$, such that for any $\eta\in(0,\eta_0]$, and $d(z_i,\Gamma_i)=O(\eta)$, there exist $\omega_{\eta,z}$ with $z=(z_1,\ldots,z_k)$, such that
\begin{align*}
\int_{\R^N}\big(\eta^{2s}(-\Delta)^{\frac s2}w_\eta(-\Delta)^{\frac s2}\psi+(1+\eta^{2s}V(x))w_\eta\psi\big)=\int_{\R^N}w_\eta^p\psi,\ \ for\ all\ \psi\in F_{\eta,z},
\end{align*}
where $w_\eta(x)=\displaystyle\sum_{i=1}^kU_{\eta,z}+\omega_{\eta,z}(x)$, and
$$
F_{\eta,z}=\Big\{\psi(x)\in H^s(\R^N):\langle v,\frac{\partial U_{\eta,z}}{\partial x_j}\rangle_\eta=0,j=1,\ldots,N,i=1,\ldots,k\Big\}.
$$
Moreover, for $N\geq\max\{2-2s,2s\}$
\begin{align}\label{pr1}
\begin{split}
\|\omega_{\eta,z}\|_\eta=&O\Big(\sum_{i=1}^k|V(z_i)-V_i|\eta^{\frac N2+2s}+\sum_{i=1}^k|\nabla V(z_i)|\eta^{\frac N2+2s+1}+\eta^{\frac N2+2s+2}\Big)\\
&+\begin{cases}O\big(\eta^{\frac N2+\frac p2(N+2s)}\big), \, &\text{if}\ 1<p\leq2\\[2mm]
O\big(\eta^{\frac N2+ (N+2s)}\big),\,&\text{if}\ p>2
\end{cases}
=O\Big(\eta^{\frac N2+2s+\gamma}\Big)
\end{split}
\end{align}
with $\gamma=\max\{0,1-2s\}$.
\end{Prop}

\begin{Rem}
The proof of the existence result is standard, and we omit the details. The last inequality in \eqref{pr1} holds because  that $N\geq\max\{2-2s,2s\}$ implies  $\frac p2(N+2s)>2s+\gamma$.

\end{Rem}

It is standard to know that, to obtain a true solution of \eqref{eq2'}, we need to choose $z$ such that
\begin{align*}
 &-\eta^{2s}\int_{\partial''\mathbf B_\rho^+(x_{a,i})}t^{1-2s}\frac{\partial\tilde \omega_{\eta,z}}{\partial\nu}\frac{\partial\tilde \omega_{\eta,z}}{\partial y^j}
+\frac12\eta^{2s}\int_{\partial''\mathbf B_\rho^+(x_{a,i})}t^{1-2s}|\nabla\tilde \omega_{\eta,z}|^2\nu_j\nonumber\\&
=-\int_{\partial B_\rho(x_{a,i})} \big(1+\eta^{2s}V(x)\big)\omega_{\eta,z}\frac{\partial  \omega_{\eta,z}}{\partial y^j}
+ \int_{\partial B_\rho(x_{a,i})} \omega_{\eta,z}^{p}\frac{\partial  \omega_{\eta,z}}{\partial y^j},
\end{align*}
which, similar to the proof of \eqref{poh2},  is equivalent to
\begin{align*}
\int_{B_\rho(x_{a,i})}\frac{\partial V(x)}{\partial x_j}w_{\eta,z}^2
=O\big(\eta^{2N}\big).
\end{align*}

We consider    $z=(z_1,\ldots,z_k)$, $z_i$ close to $b_i$, $z_i\in\Gamma_{t,i}$, for some $t$ close to $V_i$.
Denote $\nu_i$ as the unit normal vector of $\Gamma_{t,i}$ at
$z_i$, while $\tau_{i,j}$ as the principal directions of $\Gamma_{t,i}$, with $i=1,\ldots,k,\ j=1,\ldots,N-1$.
Then, at $z_i$, there holds that
$$D_{\tau_{i,j}}V(z_i)=0,~~\mbox{for}~~j=1,\ldots,N-1, ~~\mbox{and}~~|\nabla V(z_i)|=|D_{\nu_i}V(z_i)|.$$ Moreover, we have the following results.

\begin{Lem}\label{lem4.6}
Under the assumption (V), then
\begin{align*}
\int_{B_\rho(x_{a,i})}D_{\nu_i} V(x)w_{\eta,z}^2
=O(\eta^{2N})
\end{align*}
is equivalent to
\begin{align*}
D_{\nu_i} V(z_i)
=O(\eta^{2s+\gamma}).
\end{align*}
\end{Lem}

\begin{proof}
On the one hand,
\begin{align}\label{4.23}
&\int_{B_\rho(x_{a,i})}D_{\nu_i} V(x)U_{\eta,z_i}^2(x)\nonumber\\
&=-2\int_{B_\rho(x_{a,i})}D_{\nu_i} V(x)U_{\eta,z_i}(x)\omega_{\eta,z}-\int_{B_\rho(x_{a,i})}D_{\nu_i} V(x)\omega_{\eta,z}^2+O(\eta^{2N})\nonumber\\&
=O(\eta^{N+2s+\gamma}).
\end{align}

On the other hand,
\begin{align*}
&\int_{B_\rho(x_{a,i})}D_{\nu_i} V(x)U_{\eta,z_i}^2(x)
=a_*\eta^{N}D_{\nu_i} V(z_i)+O(\eta^{N+2})
\end{align*}
which, combined with \eqref{4.23}, gives that $$D_{\nu_i} V(z_i)=O(\eta^{2s+\gamma}).$$

\end{proof}

\begin{Lem}\label{lem4.7}
Under the assumption (V), then
\begin{align*}
\int_{B_\rho(x_{a,i})}D_{\tau_i} V(x)w_{\eta,z}^2
=O(\eta^{2N})
\end{align*}
is equivalent to
\begin{align*}
(D_{\tau_i} \Delta) V(z_i)
=O\Big(\sum_{i=1}^k|V(z_i)-V_i|\eta^{2s-1}+\eta^{2s}\Big).
\end{align*}
\end{Lem}

\begin{proof}
Let $H(x)=\langle\nabla V(x),\tau_{i}\rangle$.
On the one hand,
\begin{align}\label{4.24}
&\int_{B_\rho(x_{a,i})}H(x)U_{\eta,z_i}^2(x)\nonumber\\
&=-2\int_{B_\rho(x_{a,i})}H(x)U_{\eta,z_i}(x)\omega_{\eta,z}-\int_{B_\rho(x_{a,i})}H(x) \omega_{\eta,z}^2+O(\eta^{2N})\nonumber\\&
=O\left((\eta^{\frac N2+1}|\nabla H(z_i)|+\eta^{\frac N2+2})\|\omega_{\eta,z}\|_\eta+|\nabla H(z_i)|\|\omega_{\eta,z}\|_\eta^2\right)+O(\eta^{2N}).
\end{align}

On the other hand,  since $H(x_{a,i})=0$, we set
\begin{align}\label{B}
B=\frac1{N}\int_{\R^N}|x|^2U^2,
\end{align}
and then
\begin{align*}
&\int_{B_\rho(x_{a,i})}H(x)U_{\eta,z_i}^2(x)
=\frac B{2}\eta^{N+2}\Delta H(z_i)+O(\eta^{N+4})
\end{align*}
which, combined with \eqref{4.24}, gives the result.

\end{proof}

\begin{Thm}\label{thexi}
For $\lambda>0$ large, the problem \eqref{eq2} has a solution $w_\lambda$ of the form
\begin{align*}
w_\lambda(x)=\sum_{i=1}^k\lambda^{\frac1{p-1}}U\Big(\lambda^{\frac1{2s}}(x-x_{\lambda,i})\Big)+\omega_\lambda,
\end{align*}
where $x_{\lambda,i}\rightarrow b_i$ and $\displaystyle\int_{\R^N}\big(|(-\Delta)^{\frac s2}\omega_\lambda|^2+\omega_\lambda^2\big)\rightarrow0$ as $\lambda\rightarrow+\infty$.
\end{Thm}

\begin{proof}
We are sufficed to solve \eqref{eq2'}.  From Lemma \ref{lem4.6} and Lemma \ref{lem4.7}, we know that solving \eqref{eq2'} is equivalent to solving $d(z_i,\Gamma_i)=O(\eta)$,  
$$D_{\nu_i} V(z_i)
=O\big(\eta^{2s+\gamma}\big),\ \ \ (D_{\tau_i} \Delta) V(z_i)
=O\Big(\sum_{i=1}^k|V(z_i)-V_i|\eta^{2s-1}+\eta^{2s}\Big).$$

Now we take $\bar z_i\in\Gamma_i$ to be the point such that $z_i-\bar z_i=\alpha_i\nu_i$ for some $\alpha_i\in\R$. Then, we have $D_{\nu_i}V(\bar z_i)=0$.
Then,
$$D_{\nu_i} V(z_i)=D_{\nu_i} V(z_i)-D_{\nu_i} V(\bar z_i)=D^2_{\nu_i\nu_i} V(\bar z_i)\big\langle z_i-\bar z_i,\nu_i\big\rangle+O\big(|z_i-\bar z_i|^2\big).$$
By the non-degenerate assumption, it holds that
$$D_{\nu_i} V(z_i)
=O\big(\eta^{2s+\gamma}\big)\Leftrightarrow
\langle z_i-\bar z_i,\nu_i\rangle=O(\eta^{2s+\gamma}+|z_i-\bar z_i|^2),$$ which means that
 \begin{align}\label{4.25}
 |z_i-\bar z_i|=O(\eta^{2s+\gamma}) \ \Rightarrow\  d(z_i,\Gamma_i)=|z_i-\bar z_i|=O(\eta)
 .\end{align}

Let $\bar \tau_{i,j}$ be the $j$-th tangential unit vector of $\Gamma_i$ at $\bar z_i$. By the assumption (V), we have
$$ (D_{ \tau_{i,j}} \Delta) V( z_i)= (D_{\bar\tau_{i,j}} \Delta) V(\bar z_i)+O(|z_i-\bar z_i|)=(D_{\bar\tau_{i,j}} \Delta) V(\bar z_i)+O(\eta^{2s+\gamma}),$$
and
\begin{equation*}
\begin{split}
(D_{\bar\tau_{i,j}} \Delta) V(\bar z_i)=&(D_{\bar\tau_{i,j}} \Delta) V(\bar z_i)-(D_{ \tau_{i,j,0}} \Delta) V(b_i)\\=&\big\langle(\nabla_T D_{ \tau_{i,j,0}}\Delta V)(b_i),\bar z_i-b_i\big\rangle+O\big(|\bar z_i-b_i|^2\big),
\end{split}
\end{equation*}
where $\nabla _{T_i}$ is the tangential gradient on  $\Gamma_i$ at $b_i\in \Gamma_i$, and $\tau_{i,j,0}$ is the $j$-th tangential unit vector of $\Gamma_i$ at $b_i$. Therefore,
we have $$(D_{ \tau_{i,j}} \Delta) V( z_i)=O\Big(\sum_{i=1}^k|V(z_i)-V_i|\eta^{2s-1}+\eta^{2s}\Big),$$ which implies that
\begin{align}\label{4.26}
\Big\langle(\nabla_T D_{ \tau_{i,j,0}}\Delta V)(b_i),\bar z_i-b_i\Big\rangle=O\Big(\sum_{i=1}^k|V(z_i)-V_i|\eta^{2s-1}+\eta^{2s}\Big)=O\Big(\eta^{2s}\Big).
\end{align}

Hence, we can solve \eqref{4.25} and \eqref{4.26} to obtain $z_i=x_{\eta,i}$ with $x_{\eta,i}\rightarrow b_i$.
\end{proof}

Now we are in a position to prove Theorem \ref{th3}.
\begin{proof}[\textbf{Proof of Theorem \ref{th3}}]
Let $w_\lambda$ be a peak solution obtained in Theorem \ref{thexi}. Define $u_\lambda=w_\lambda/\Big(\displaystyle\int_{\R^N}w_\lambda^2\Big)^{\frac12}$. Then
$\displaystyle\int_{\R^N}u_\lambda^2=1$ and
\begin{align}
(-\Delta)^su_\lambda+V(x)u_\lambda=a_\lambda u_\lambda^p-\lambda u_\lambda,\ \ a_\lambda=\int_{\R^N}w_\lambda^2.
\end{align}

Similar to the proof of Proposition \ref{prop3.4}, there holds that
\begin{align*}
\lambda^{\frac N{2s}-\frac2{p-1}}\Big(\int_{\R^N}w_\lambda^2\Big)^{\frac2{p-1}}=ka_*+o(1),\ as\ a\rightarrow a_0
\end{align*}
with $a_0=0$ if $\frac N{2s}-\frac2{p-1}>0$, $a_0=(ka_*)^{\frac{p-1}2}$ if $\frac N{2s}-\frac2{p-1}=0$, while $a_0=+\infty$ if $\frac N{2s}-\frac2{p-1}<0$.

\vskip 0.2cm

Take $\lambda_0>0$ large and $b_0=\displaystyle\int_{\R^N}w_{\lambda_0}^2$. Then, if $\frac N{2s}-\frac2{p-1}>0$, for any $a\in(0,b_0)$, by the mean value theorem, there exists some large $\lambda_a$,
such that the solution $u_a$ to \eqref{eq1} with $\lambda=\lambda_a$ satisfies $\displaystyle\int_{\R^N}w_\lambda^2=a$, and for such $a$, we obtain a $k$-peak solution
for \eqref{eq2}, with $\mu_a=-\lambda_a$.

Similarly, if  $\frac N{2s}-\frac2{p-1}=0$, for any $a$ between $b_0$ and $(ka_*)^{\frac{p-1}2}$, while if $\frac N{2s}-\frac2{p-1}<0$, for any
 $a>b_0$ large enough,  there exists $\lambda_a$,
such that the solution $u_a$ to \eqref{eq1} with $\lambda=\lambda_a$ satisfies $\displaystyle\int_{\R^N}w_\lambda^2=a$, and for such $a$, we obtain a $k$-peak solution
for \eqref{eq2}, with $\mu_a=-\lambda_a$.

\end{proof}

\subsection{Clustering peak solutions}$\,$  \vskip 0.2cm

In the end of this section, we are concerned with the clustering multi-peak solutions to problem \eqref{eq1}--\eqref{eqnorm}, and sketch the proof of Theorem \ref{th3'}.

\vskip 0.2cm

The function $\Delta V(x)|_{x\in \Gamma_{i}}$ has a minimum point and a maximum point.  Let us assume that  the function $\Delta V(x)|_{x\in\Gamma_{i_0}}$ has isolated maximum point $b\in\Gamma_{i_0}$, that is  $\Delta V(x)<\Delta V(b)$
for all $x\in\Gamma_{i_0}\cap(B_\delta(b)\setminus b)$. Let $x_{\eta,j}\rightarrow b, j=1,\ldots,k$, $\frac{|x_{\eta,i}-x_{\eta,j}|}\eta\rightarrow+\infty, i\neq j$
as $\eta\rightarrow 0$.
We take $\displaystyle\sum_{j=1}^kU_{\eta,x_{\eta,j}}$ as the approximate solution of  \eqref{eq2'} and sufficed to show the following result.

\begin{Prop}
Assume  \textup{($V$)} and $\frac{\partial^2V(x)}{\partial\nu_{i_0}^2}\neq0$ for any $x\in\Gamma_{i_0}$ with some $i_0\in\{1,\ldots,k\}$.
If $b\in \Gamma_{i_0}$ is an isolated maximum point of $ \Delta V(x)|_{x\in \Gamma_{i_0}}$ on $\Gamma_{i_0}$, then
  for any integer $k>0$, there exists an $\eta_0>0$, such that for any $\eta\in (0, \eta_0]$,  problem \eqref{eq2'} has a solution $w_\eta$ of the form
\begin{align*}
w_\eta (x) =\sum_{j=1}^kU_{\eta,x_{\eta,j}} +\omega_\eta,
\end{align*}
where  $x_{\eta, j}\to b$, $j=1,\cdots, k$,  $\frac{|x_{\eta, i}- x_{\eta, j}|}\eta \to +\infty$, $i\neq j$,  and $\displaystyle\int_{\mathbb R^N} \bigl(
\eta^{2s} |(-\Delta)^{\frac s2} \omega_\eta|^2 + \omega_\eta^2\bigr)= o(\eta^{\frac N{2}})$ as $\eta\to 0$.
\end{Prop}

\begin{proof}
Define the energy functional
\[
I(u) = \frac12 \int_{\mathbb R^N} \Bigl( \eta^{2s} |(-\Delta)^{\frac s2} u|^2  +
\big(1+ \eta^{2s}V(x) \big)u^2\Bigr) -\frac1{p+1} \int_{\mathbb R^N} u^{p+1}.
\]
We obtain the energy expansion:

\begin{equation}\label{I}
\begin{split}
&I\Bigl(  \sum_{j=1}^kU_{\eta,x_{\eta,j}} \Bigr)
\\=& (1+\eta^{2s}V_{i_0})^{-\frac N{2s}}\Big(k  E_1 \eta^N + E_2 \eta^{N+2s}\sum_{j=1}^k  \frac{\partial^2 V(\bar x_{\eta,j})}{\partial \nu_{i,j}^2} r_{\eta,j}^2+E_3\eta^{N+2s+2} \sum_{j=1}^k \Delta V(\bar x_{\eta,j})\\
& - \sum_{j>m}(c_0+o(1))\eta^{N} \Bigl( \frac{\eta}{ |x_{\eta,m} -x_{\eta,j}|}
 \Bigr)^{N+2s}\\&
+O\Big(\eta^{N+2s+3}+\eta^{N}\sum_{j>m} \bigl( \frac{\eta}{ |x_{\eta,m} -x_{\eta,j}|}
 \bigr)^{N+2s+\tau}+\eta^{N+2s}r_{\eta,j}^3)\Big),
 \end{split}
\end{equation}
where $E_l,c_0>0,l=2,3$, $$E_1= (\frac12-\frac1{p+1})
\displaystyle\int_{\mathbb R^N}U^{p+1}>0~~\mbox{and}~~r_{\eta,j}= |x_{\eta,j}-\bar x_{\eta,j}|,$$
with $\bar x_{\eta,j}\in \Gamma_{i_0}$ the point such that $|x_{\eta,j}-\bar x_{\eta,j}|= d(x_{\eta,j}, \Gamma_{i_0})$.

\vskip 0.1cm

In fact, it is easy to get
\begin{equation}\label{I'}
\begin{split}
&I\Bigl(  \sum_{j=1}^kU_{\eta,x_{\eta,j}} \Bigr)
\\=& (1+\eta^{2s}V_{i_0})^{-\frac N{2s}}\Big(k (\frac12-\frac1{p+1})
\displaystyle\int_{\mathbb R^N}U^{p+1}+ E_2' \eta^{N+2s}\sum_{j=1}^k\big(V(x_{\eta,j})-V(\bar x_{\eta,j})\big)\\&+E_3'\eta^{N+2s+2} \sum_{j=1}^k \Delta V( x_{\eta,j}) - \sum_{j>m}\big(c_0+o(1)\big)\eta^{N} \Bigl( \frac{\eta}{ |x_{\eta,m} -x_{\eta,j}|}
 \Bigr)^{N+2s}\\&
+O(\eta^{N+2s+3}+\eta^{N}\sum_{j>m} \Bigl( \frac{\eta}{ |x_{\eta,m} -x_{\eta,j}|}
 \Bigr)^{N+2s+\tau})\Big),
 \end{split}
\end{equation}
and
\begin{equation}\label{c33-3-1}
 V(x_{\eta,j}) = 1 +\frac{\partial^2 V(\bar x_{\eta,j})}{\partial \nu_{i,j}^2} r_{\eta,j}^2 + O(r_{\eta,j}^3) ~
\mbox{and}~
\Delta V(x_{\eta,j})=\Delta V(\bar x_{\eta,j})+ O(r_{\eta,j}).
\end{equation}
So,  \eqref{I}  follows from  \eqref{I'} and \eqref{c33-3-1}.

\vskip 0.1cm

To obtain a solution $w_{\eta}$ of the form  $\displaystyle\sum_{j=1}^kU_{\eta,x_{\eta,j}} +\omega_{\eta }$,   we can first carry out the reduction argument as in Proposition~\ref{propeta} to obtain $\omega_{\eta }$, satisfying

\begin{equation}\label{clustererr}
\|\omega_{\eta}\|_\eta
=O\Big(\eta^{\frac N2+2s+1}+\eta^{\frac N2}\sum_{i\neq j}\big(\frac\eta{|x_{\eta,i}-x_{\eta,j}|}\big)^{\min\{\frac p2,1\}(N+2s)}\Big).
\end{equation}
  Define

\[
K(x_{\eta,1},\cdots, x_{\eta,k})=I\bigl(U_{\eta,x_{\eta,j}} +\omega_{\eta}\bigr).
\]
From \eqref{clustererr}, we get the same expansion \eqref{I}  for $K(x_{\eta,1},\cdots, x_{\eta,k})$.
Now we set
\[
\mathcal B= \bigl\{\bigl( r, \bar x):  r\in (-\delta \eta^{\frac12}, \delta \eta^{\frac12}),\; \bar x\in  B_\delta(b_i)\cap \Gamma_i\bigr\}.
\]

If $\Gamma_{i_0}$ is a local maximum set of $V(x)$  and $\frac{\partial^2 V(\bar x)}{\partial \nu_{i_0}^2}<0$ for any $\bar x\in \Gamma_{i_0}$, then it is easy to
prove that $K(x_{\eta,1},\cdots, x_{\eta,k})$
has a critical point, which is a maximum point of $K$ in
$$
S_{\eta, k}:=\Bigl\{ (x_{\eta,1},\cdots, x_{\eta,k}):  x_{\eta,j}\in \mathcal B,~~
\frac{ |x_{\eta, j}-x_{\eta, m}|}{\eta}> \eta^{\sigma-1},\, m\ne j\Bigr\},
 $$  where   $\sigma\in(\frac12,\frac{N-1}{N+2s})$ is some constant 
 such that both $\sigma>\frac12$ and $(1-\sigma)(N+2s)>2s+1$ hold.

\vskip 0.1cm

If $\Gamma_{i_0}$ is a local minimum set of $V(x)$  and $\frac{\partial^2 V(\bar x)}{\partial \nu_{i_0}^2}>0$ for any $\bar x\in \Gamma_{i_0}$, then  $$E_2 \frac{\partial^2 V(\bar x)}{\partial \nu_{i_0}^2}r^2+  E_3 \eta^{2}  \Delta V(\bar x)$$ has
a saddle point $(0, b)$ in $\mathcal B$.  We can use a topological argument as in \cite{DY10,DY11} to prove that   $K(x_{\eta,1},\cdots, x_{\eta,k})$
has a critical point  in $S_{\eta, k}$.

\end{proof}

\section{Local Uniqueness}

\subsection{Precise estimates by Pohozaev identities}$\,$  \vskip 0.2cm

From Lemma \ref{lemla}, we know that a $k$-peak solution $v_a$ to equation \eqref{eqv} can be written as
\begin{align*}
v_a(x)=\sum_{i=1}^kU_{\varepsilon,x_{a,i}}(x)+\varphi_a(x),
\end{align*}
where $|x_{a,i}-b_i|=o(1), d(x_{a,i},\Gamma_i)=O(\varepsilon), \varepsilon=(-\mu_a)^{-\frac1{2s}}$,
$$U_{\varepsilon,x_{a,i}}(x)=\big(1+\varepsilon^{2s}V_i\big)^{\frac1{p-1}}U
\Big(\frac{(1+\varepsilon^{2s}V_i)^{\frac1{2s}}(x-x_{a,i})}{\varepsilon}\Big),$$
 $\varphi_a\in \displaystyle\bigcap_{i=1}^k E_{a,x_{a,i}}$. For $N\geq 2s+2$, it holds that $\frac p2(N+2s)>2s+1$, which implies then
\begin{align}\label{phia}
\|\varphi_a\|_a=\|\varphi_a\|_\varepsilon=&O\Big( |\sum_{i=1}^k\big(V(x_{a,i})-V_i)\big)|\varepsilon^{\frac N2+2s}+|\sum_{i=1}^k\nabla V(x_{a,i})|\varepsilon^{\frac N2+2s+1}+\varepsilon^{\frac N2+2s+1}\Big).
\end{align}
Moreover, $x_{a,i}\in\Gamma_{t_a,i}$ with some $t_a\rightarrow V_i$. Denote $\nu_{a,i}$ and $\tau_{a,i,j}$ as the unit normal vector
and the principal direction of $\Gamma_{t_a,i}$ at $x_{a,i}$ respectively.
Then at $x_{a,i}$, there hold that
$$D_{\tau_{a,i,j}}V(x_{a,i})=0,\ \ |\nabla V(x_{a,i})|=|D_{\nu_{a,i}}V(x_{a,i})|.$$

\begin{Lem}
Under the assumption (V), it holds that
\begin{align*}
D_{\nu_{a,i}}V(x_{a,i})=O\Big(\varepsilon^{\min\{2,2s+1\}}\Big).
\end{align*}

\end{Lem}
\begin{proof}
Using \eqref{poh2}, we obtain that
\begin{align*}
\int_{B_\rho(x_{a,i})}D_{\nu_{a,i}} V(x)v_a^2
=O(\varepsilon^{2N}).
\end{align*}

On the one hand,
\begin{equation}\label{5.4}\begin{split}
&\int_{B_\rho(x_{a,i})}D_{\nu_{a,i}} V(x)U_{\varepsilon,x_{a,i}}^2(x) \\
&=-2\int_{B_\rho(x_{a,i})}D_{\nu_{a,i}} V(x)U_{\varepsilon,x_{a,i}}(x)\varphi_{a}-\int_{B_\rho(x_{a,i})}D_{\nu_{a,i}} V(x)\varphi_{a}^2
+O(\varepsilon^{2N}) \\&
=O\big(\varepsilon^{\frac N2}|D_{\nu_{a,i}} V(x_{a,i})|\|\varphi_{a}\|_a+\varepsilon^{\frac N2+1}\|\varphi_{a}\|_a+\|\varphi_{a}\|_a^2\big)+O(\varepsilon^{2N}) \\&
=O\Big(\sum_{i=1}^k|V(x_{a,i})-V_i|\varepsilon^{N+2s}+|\sum_{i=1}^k \nabla V(x_{a,i})|\varepsilon^{N+2s+1}+\varepsilon^{N+2s+2}\Big)=O(\varepsilon^{N+2s+1}).
\end{split}\end{equation}

On the other hand, using the definition \eqref{B}
\begin{equation}\label{5.5}\begin{split}
&\int_{B_\rho(x_{a,i})}D_{\nu_{a,i}} V(x)U_{\varepsilon,x_{a,i}}^2(x) \\&
=a_*\varepsilon^{N}(1+\varepsilon^{2s}V_i)^{\frac2{p-1}}D_{\nu_{a,i}} V(x_{a,i})+\frac B2\varepsilon^{N+2}\Delta D_{\nu_{a,i}}V(x_{a,i})+O(\varepsilon^{N+4}),
\end{split}\end{equation}
which, combined with \eqref{5.4}, gives that $$D_{\nu_{a,i}} V(x_{a,i})=O\big(\varepsilon^{\min\{2,2s+1\}}\big).$$

\end{proof}

Now we take $\bar x_{a,i}\in\Gamma_{i}$ be the point such that $x_{a,i}-\bar x_{a,i}=\alpha_{a,i}\nu_{a,i}$ for some $\alpha_{a,i}\in\R$.

 \begin{Lem}\label{lem5.2}
 Under the condition (V), we have that
 \begin{equation}\label{xb}
 \begin{cases}
  \bar x_{a,i}-b_i=L_i\varepsilon^{2}+O(\varepsilon^{2s+2}),\\[2mm]
  x_{a,i}-\bar x_{a,i}=-\frac B{2a_*}\Delta D_{\nu_{a,i}}V(b_i)\Big(\frac{\partial^2V(b_i)}{\partial \nu_i^2}\Big)^{-1}\varepsilon^{2}+
 O(\varepsilon^{2s+2}),
 \end{cases}
 \end{equation}
where B is as in \eqref{B}, $L_{i}$ is a vector depending on $b_{i}$ and $i=1,...,k.$
 \end{Lem}

 \begin{proof}
 From \eqref{5.5}, and
 \begin{align*}
&\int_{B_\rho(x_{a,i})}D_{\nu_{a,i}} V(x)U_{\varepsilon,x_{a,i}}^2(x)\nonumber\\
&=-2\int_{B_\rho(x_{a,i})}D_{\nu_{a,i}} V(x)U_{\varepsilon,x_{a,i}}(x)\varphi_{a}-\int_{B_\rho(x_{a,i})}D_{\nu_{a,i}} V(x)\varphi_{a}^2
+O(\varepsilon^{2N})\nonumber\\&
=O\Big(\varepsilon^{\frac N2}|D_{\nu_{a,i}} V(x_{a,i})|\cdot\|\varphi_{a}\|_a+\varepsilon^{\frac N2+1}\|\varphi_{a}\|_a+\|\varphi_{a}\|_a^2\Big)+O(\varepsilon^{2N})\nonumber\\&
=O\big(\varepsilon^{N+2s+2}\big),
\end{align*}
we obtain that
 \begin{align*}
\big(a_*+ &O(\varepsilon^{2s})\big)D_{\nu_{a,i}} V(x_{a,i})+\frac B2\varepsilon^{2}\Delta D_{\nu_{a,i}}V(x_{a,i})
\\ =&O\Big(\varepsilon^{2s+2}+\varepsilon^{2s+1}\sum_{i=1}^k|V(x_{a,i})-V_i|\Big)\\=&
O\Big(\varepsilon^{2s+2}+\varepsilon^{2s+1}\sum_{i=1}^k|x_{a,i}-\bar x_{a,i}|^2\Big).
 \end{align*}

 Since $\frac{\partial^2V(b_i)}{\partial \nu_i^2}\neq0$, the outward vector $\nu_{a,i}(x)$ and the tangential unit vector $\tau_{a,i}(x)$ of $\Gamma_{a,i}$ at
 $x_{a,i}$ are Lip-continuous in $W_{\delta,i}$, we obtain from the above that
 \begin{align}\label{x-x}
 x_{a,i}-\bar x_{a,i}=-\frac B{2a_*}\Delta D_{\nu_{a,i}}V(b_i)\Big(\frac{\partial^2V(b_i)}{\partial \nu_i^2}\Big)^{-1}\varepsilon^{2}+
 O\Big(\varepsilon^{2s+2}+\varepsilon^{2s+1}\sum_{i=1}^k|x_{a,i}-\bar x_{a,i}|^2\Big).
 \end{align}
From \eqref{phia} and \eqref{x-x}, for  $\frac{4s+4}{N+2s}<p<\frac{N+2s}{N-2s}$ (which can be satisfied by $N\geq 2s+4$),
\begin{align}\label{varphi}
\|\varphi_a\|_a&=O\Big(\sum_{i=1}^k|x_{a,i}-\bar x_{a,i}|\varepsilon^{\frac N2+2s}+\varepsilon^{\frac N2+4s+2}+\varepsilon^{\frac N2+2s+2}\Big)
=O\big(\varepsilon^{\frac N2+2s+2}\big).
\end{align}
Let $H(x)=\langle\nabla V(x),\tau_{a,i}\rangle$.
Then for $N\geq2s+4$,
\begin{align}\label{5.11}
&\int_{B_\rho(x_{a,i})}H(x)U_{\varepsilon,x_{a,i}}^2(x)\nonumber\\
&=-2\int_{B_\rho(x_{a,i})}H(x)U_{\varepsilon,x_{a,i}}(x)\varphi_{a}-\int_{B_\rho(x_{a,i})}H(x) \varphi_a^2\nonumber\\&
=-2\int_{B_\rho(x_{a,i})}H(x)U_{\varepsilon,x_{a,i}}(x)\varphi_{a}+O(\|\varphi_a\|_a^2)+O(\varepsilon^{2N})\nonumber\\&
=-2\int_{B_\rho(x_{a,i})}\langle\nabla H(x_{a,i}),x-x_{a,i}\rangle U_{\varepsilon,x_{a,i}}(x)\varphi_{a}+O(\varepsilon^{N+2s+4})
\end{align}

On the one hand,  since $\nabla V=0$, for $x\in\Gamma_i$,
\begin{align}\label{5.12}
\nabla H(x_{a,i})=\big\langle\nabla^2V(x_{a,i}),\tau_{a,i}\big\rangle=\big\langle\nabla^2V(\bar x_{a,i}),\bar\tau_{a,i}\big\rangle+O\big(|x_{a,i}-\bar x_{a,i}|\big)
=O\big(|x_{a,i}-\bar x_{a,i}|\big),
\end{align}
where $\bar x_{a,i}\in\Gamma_i$ is the point such that $ x_{a,i}-\bar  x_{a,i}=\beta_{a,i}\nu_{a,i}$ for some $\beta_{a,i}\in\R$ and $\bar\tau_{a,i}$ is the tangential
vector of $\Gamma_i$ at $\bar x_{a,i}\in\Gamma_i$. Hence,
\begin{align*}
&\int_{B_\rho(x_{a,i})}\langle\nabla H(x_{a,i}),x-x_{a,i}\rangle U_{\varepsilon,x_{a,i}}(x)\varphi_{a}\\
&=O\big(\varepsilon^{\frac N2+1}|\nabla H(x_{a,i})|\cdot \|\varphi_{a}\|_a\big)=O\big(\varepsilon^{N+2s+3}|x_{a,i}-\bar x_{a,i}|\big)=O\big(\varepsilon^{N+2s+5}\big).
\end{align*}
Then by \eqref{5.11} and \eqref{5.12}, it follows
\begin{align}\label{5.13}
&\int_{B_\rho(x_{a,i})}H(x)U_{\varepsilon,x_{a,i}}^2(x)=O\big(\varepsilon^{N+2s+4}\big).
\end{align}

On the other hand,  applying the Taylor's expansion, we obtain that

\begin{equation}\label{5.14}\begin{split}
 \int_{B_\rho(x_{a,i})}H(x)U_{\varepsilon,x_{a,i}}^2(x)=&
\frac {B\varepsilon^{N+2}}2\big(1+V_i\varepsilon^{2s}\big)(D_{\tau_{a,i}}\Delta V)(x_{a,i})\\&+\frac{H_{\tau_i}\varepsilon^{N+4}}{24}+
O\big(\varepsilon^{N+2s+4}\big),
\end{split}\end{equation}
where $$H_{\tau_i}=\sum_{l=1}^2\sum_{m=1}^2\frac{\partial^4H(b_i)}{\partial x_l^2\partial x_m^2}\int_{\R^N}x_l^2x_m^2U^2.$$

Combining \eqref{5.13} and \eqref{5.14}, we obtain that
\begin{align}\label{5.15}
&(D_{\tau_{a,i}}\Delta V)(x_{a,i})=-\frac{H_{\tau_i}\varepsilon^{2}}{12B}+
O(\varepsilon^{2s+2}).
\end{align}
By \eqref{x-x}, we get
\begin{equation}\label{5.16a}\begin{split}
\big(D_{ \tau_{a,i}} \Delta\big) V(x_{a,i})&= \big(D_{\bar\tau_{a,i}} \Delta\big) V(\bar x_{a,i})+\langle A_{\tau_i},x_{a,i}-\bar x_{a,i}\rangle+O\big(|x_{a,i}-\bar x_{a,i}|^2\big) \\&=
(D_{\bar\tau_{a,i}} \Delta) V(\bar x_{a,i})+B_{\tau_i}\varepsilon^2+O\big(\varepsilon^{2s+2}\big),
\end{split}\end{equation}
where $A_{\tau_i}$ is a vector depending on $b_i$ and $B_i$ is a constant depending on $b_i$.
Moreover,
\begin{align}\label{5.16}
\big(D_{\bar\tau_{a,i}} \Delta\big) V(\bar x_{a,i})=\big((D^2_{ \tau_{a,i}} \Delta) V(b_i) \big)(\bar x_{a,i}-b_i)+O\big(|\bar x_{a,i}-b_i|^2\big).
\end{align}
Therefore, by \eqref{5.15}, \eqref{5.16a} and \eqref{5.16}, there hold that
\begin{align*}
\big((D^2_{ \tau_{a,i}} \Delta) V(b_i) (x_{a,i})\big)(\bar x_{a,i}-b_i)=-\big(\frac{H_{\tau_i}}{12B}+B_{\tau_{a,i}}\big)\varepsilon^{2}+
O\big(\varepsilon^{2s+2}\big)+O\big(|\bar x_{a,i}-b_i|^2\big).
\end{align*}
Since $\big(D^2_{ \tau_{a,i}} \Delta\big) V(b_i) $ is non-singular, we conclude the proof.
 \end{proof}

\smallskip

\subsection{Local uniqueness} $\,$  \vskip 0.2cm

\smallskip 
We assume $\frac{4s+4}{N+2s}<p<\frac{N+2s}{N-2s}$, which can be satisfied by $N\geq 2s+4$. Set
\begin{align*}
\delta_a=\begin{cases}\Big(\frac{a^{\frac2{p-1}}}{ka_*}\Big)^{\frac1{N-\frac{4s}{p-1}}},\,\,\,& p-1\neq\frac{4s}N,
\\[2mm]
\Big(\frac{2s}{(s+2)B_1}|ka_*-a^{\frac2{p-1}}|\Big)^{\frac1{2(s+1)}},\,\,\,& p-1=\frac{4s}N,
\end{cases}
\end{align*}where $B_1=\frac1N\displaystyle\sum_{i=1}^k\Delta V(b_i)\displaystyle\int_{\R^N}|x|^2U^2(x).$
\begin{Prop}
Under the assumption (V), there hold that
\begin{align}\label{pohdelta}
-\mu_a\delta_a^{2s}=1+\gamma_1\delta_a^{2s}+O(\delta_a^{2s+2}),
\end{align}
and
\begin{align}\label{pohxb}
x_{a,i}-b_i=\bar L_i\delta_a^{2}+O(\delta_a^{2s+2}),\ \ i=1,\ldots,k,
\end{align}
where $\gamma_1$ and the vector $\bar L_i$ are constants.
\end{Prop}

\begin{proof}
(a) Firstly,  if $p-1\neq\frac{4s}N$, we follow
 the proof of Proposition \ref{prop3.4} to find
\begin{align*}
\frac{a^{\frac2{p-1}}}{(-\mu_a)^{\frac2{p-1}-\frac N{2s}}}&=
\sum_{i=1}^k(1+\frac{V_i}{-\mu_a})^{\frac2{p-1}-\frac N{2s}}\int_{\R^N}U^2+
O(|\sum_{i=1}^k(V(x_{a,i})-V_i)|(-\mu_a)^{-1}\\&
\quad+|\sum_{i=1}^k\nabla V(x_{a,i})|(-\mu_a)^{-1-\frac1{2s}}+(-\mu_a)^{-1-\frac1s})
+\begin{cases}O((-\mu_a)^{-\frac{pN}{4s}-\frac p2}),\,&\text{ if}\,\, 1<p\leq2,\\O((-\mu_a)^{-\frac N{2s}-1}),\,&\text{ if}\,\, p>2,\end{cases}
\\
&\quad+O\Big((-\mu_{a})^{-{\frac{N}{2s}-1}}|x_{a,j}-x_{a,i}|^{-(N+2s)}\Big),
\end{align*}
which implies that
\begin{align*}
\frac{a^{\frac2{p-1}}}{ka_*(-\mu_a)^{\frac2{p-1}-\frac N{2s}}}&=1+
\sum_{i=1}^k(\frac2{p-1}-\frac N{2s})\frac{V_i}{-\mu_a}+
O\Big(|\sum_{i=1}^k(V(x_{a,i})-V_i)|(-\mu_a)^{-1}\\&
\quad+|\sum_{i=1}^k\nabla V(x_{a,i})|(-\mu_a)^{-1-\frac1{2s}}+(-\mu_a)^{-1-\frac1s}\Big)
+\begin{cases}O((-\mu_a)^{-\frac{pN}{4s}-\frac p2}), &\text{ if}\, 1<p\leq2,\\O((-\mu_a)^{-\frac N{2s}-1}),&\text{ if}\,p>2.\end{cases}
\\
&\quad+O\Big((-\mu_{a})^{-{\frac{N}{2s}-1}}|x_{a,j}-x_{a,i}|^{-(N+2s)}\Big).
\end{align*}
Set $\delta_a=(\frac{a^{\frac2{p-1}}}{ka_*})^{\frac1{N-\frac {4s}{p-1}}}.$ Then we obtain \eqref{pohdelta}.
Finally, \eqref{pohxb} can be obtained by  \eqref{pohdelta} and Lemma \ref{lem5.2}.

\medskip
(b) Secondly, we turn to deal with the case of  $p-1=\frac{4s}N$. From \eqref{xb}, it holds that
$$x_{a,i}-b_i=\bar L_i(-\mu_a)^{-\frac1s}+O((-\mu_a)^{-\frac1s-1}),$$ with some vector $\bar L_i$.
Moreover, in view of \eqref{varphi}, $\|\varphi_a\|_a=O(\varepsilon^{\frac N2+2s+2})$.

We know that $u_a$ is the $k$-peak solutions of \eqref{eq1}-\eqref{eqnorm},      and        
we have the following Pohozaev identity:
\begin{align*}
\begin{split}
&-\int_{\partial'' {\bf B}_\rho(x_{a,i} )}t^{1-2s}\frac{\partial\tilde u_a }{\partial\nu}\langle Y-X_{a,i},\nabla\tilde u_a \rangle
+\frac{1}2\int_{\partial'' {\bf B}_\rho(x_{a,i} )}t^{1-2s}|\nabla\tilde u_a |^2\langle Y-X_{a,i},\nu\rangle
\\
&-\frac{N-2s}2\int_{\partial'' {\bf B}_\rho(x_{a,i} )}t^{1-2s}\frac{\partial\tilde u_a }{\partial\nu}\tilde u_a \\&
-\int_{\partial B_\rho(x_{a,i} )}\Big(\frac{(u_a )^{p+1}}{p+1}-\frac{-\mu_a +V(x)}{2}(u_a )^{2}
\Big)\langle y-x_{a,i},\nu\rangle
\\
 =&\int_{B_\rho(x_{a,i} )}\Big(\frac{N-2s}2-\frac N{p+1}\Big)|u_a |^{p+1}+s(-\mu_a +V(x))(u_a )^{2}\\&
 +\frac{1}2\langle y-x_{a,i} ,\nabla V(x)\rangle
 (u_a )^2,
\end{split}
\end{align*}
which can be rewritten as
\begin{align}\label{poh3}
\begin{split}
 & \int_{B_\rho(x_{a,i} )}sV(x)(u_a )^{2}
 +\frac{1}2\langle y-x_{a,i} ,\nabla V(x)\rangle
 (u_a )^2\\
 &=\int_{B_\rho(x_{a,i} )}s\mu_a (u_a )^{2}+a\Big(\frac N{p+1}-\frac{N-2s}2\Big)|u_a |^{p+1}+ \int_{\partial'' \bf{B}_{\rho}(x_{a,i} )}W_1
 +\int_{\partial B_\rho(x_{a,i} )}W_1,
\end{split}
\end{align}
where
\begin{align*}
W_1 =-t^{1-2s}\frac{\partial\tilde u_a }{\partial\nu}\langle Y-X_{a,i},\nabla\tilde u_a \rangle
+\frac{1}2t^{1-2s}|\nabla\tilde u_a |^2\langle Y-X_{a,i},\nu\rangle
-\frac{N-2s}2t^{1-2s}\frac{\partial\tilde u_a }{\partial\nu}\tilde u_a ,
\end{align*}
\begin{align*}
W_2 =\Big(-\frac{a(u_a )^{p+1}}{ p+1}+\frac{(-\mu_a +V(x))}{2}(u_a )^{2}
\Big)\langle y-x_{a,i},\nu\rangle.
\end{align*}
Recall that
\begin{align*}
u_a(x)=\sum_{i=1}^k(\frac{-\mu_a+V_i}{a})^{\frac{1}{p-1}}\left(
U\big((-\mu_a+V_i)^{\frac1{2s}}(x-x_{a,i})\big)+\varphi_a\right).
\end{align*}
Then,
\begin{align*}
\begin{split}
 & \int_{B_\rho(x_{a,i} )}sV(x)(u_a )^{2}
 +\frac{1}2\langle y-x_{a,i} ,\nabla V(x)\rangle
 (u_a )^2\\
 &=(\frac{-\mu_a+V_i}{a})^{\frac2{p-1}}\int_{B_\rho(x_{a,i} )}\left(s(V(x)-V_i)
 +\frac{1}2\langle y-x_{a,i} ,\nabla V(x)\rangle\right)
 U^2((-\mu_a+V_i)^{\frac1{2s}}(x-x_{a,i}))\\&
 \quad+sV_i\int_{B_\rho(x_{a,i} )}u_a^2+O\big((-\mu_a)^{-2-\frac 2s}\big)\\
 &=\frac{\frac s2+1}{(-\mu_a+V_i)^{\frac1s}}\Delta V(x_{a,i})\int_{\R^N}|x|^2U^2(x)+sV_i\frac{a_*}{a^{\frac2{p-1}}}+o\big((-\mu_a+V_i)^{-\frac1s}\big).
\end{split}
\end{align*}
Moreover, we  calculate more precisely that for some constants $c_l$, $l=1,2,3$,
$$\Delta V(x_{a,i})=\Delta V(b_i)+c_1(-\mu_a)^{-\frac1s}+O\big((-\mu_a)^{-\frac1s-1}\big).$$
Hence,
\begin{align*}
\begin{split}
 & \int_{B_\rho(x_{a,i} )}sV(x)(u_a )^{2}
 +\frac{1}2\langle y-x_{a,i} ,\nabla V(x)\rangle
 (u_a )^2\\
 &=\frac{\frac s2+1}{(-\mu_a+V_i)^{\frac1s}}\Delta V(b_i)\int_{\R^N}|x|^2U^2(x)+sV_i\frac{a_*}{a^{\frac2{p-1}}}+c_2(-\mu_a)^{-\frac1s-1}
 +O\big((-\mu_a)^{-\frac2s-1}\big).
\end{split}
\end{align*}
Now summing \eqref{poh3} from $i=1$ to $i=k$, since $\frac N{p+1}-\frac{N-2s}2=\frac{2s^2}{N+2s}>0$,  we have
\begin{align}\label{5.21}
\begin{split}
 &s\mu_a+\frac{2as^2}{N+2s}\int_{\R^N}|u_a|^{p+1}\\
 &=\sum_{i=1}^k\frac{\frac s2+1}{(-\mu_a)^{\frac1s}}\Delta V(b_i)\int_{\R^N}|x|^2U^2(x)+\sum_{i=1}^ksV_i\frac{a_*}{a^{\frac2{p-1}}}+c_2(-\mu_a)^{-\frac1s-1}+
 O\big((-\mu_a)^{-\frac2s-1}\big).
\end{split}
\end{align}

On the other hand, by the orthogonality $\displaystyle \int_{\R^N} U^p\big((-\mu_a+V_i)^{\frac1{2s}}(x-x_{a,i})\big)\varphi_a=0$ for $i=1,\ldots,k$, there holds that
\begin{align*}
\begin{split}
\int_{\R^N}u_a^{p+1}&=\sum_{i=1}^k(\frac{-\mu_a+V_i}{a})^{\frac{p+1}{p-1}}
\Big((-\mu_a+V_i)^{-\frac N{2s}}\int_{\R^N}U^{p+1}\\
&\quad\quad+(p+1)\int_{\R^N} (U^p((-\mu_a+V_i)^{\frac1{2s}}(x-x_{a,i}))\varphi_a+O(|\varphi_a|^{p+1}))
\Big)\\&
=\frac{N+2s}{2s}\frac{a_*}{a^{\frac{p+1}{p-1}}}\sum_{i=1}^k(-\mu_a+V_i)^{\frac{p+1}{p-1}-\frac N{2s}}
+O(\mu_a^{\frac{p+1}{p-1}}\|\varphi_a\|_a^2)\\&
=\frac{N+2s}{2s}\frac{a_*}{a^{\frac{p+1}{p-1}}}\sum_{i=1}^k(-\mu_a+V_i)^{\frac{p+1}{p-1}-\frac N{2s}}
+O((-\mu_a)^{-1-\frac2s})\\&
=\frac{N+2s}{2s}\frac{a_*}{a^{\frac{p+1}{p-1}}}\sum_{i=1}^k(-\mu_a+V_i)
+O\big((-\mu_a)^{-1-\frac2s}\big),
\end{split}
\end{align*}
which, combined with \eqref{5.21}, gives then
\begin{align*}
\begin{split}
 &s(-\mu_a)^{1+\frac1s}(ka_*-a^{\frac2{p-1}})=(\frac s2+1)\sum_{i=1}^k\Delta V(b_i)\int_{\R^N}|x|^2U^2(x)+c_3(-\mu_a)^{-1}+O\big((-\mu_a)^{-\frac1s-1}\big).
\end{split}
\end{align*}
If we set $\delta=\left(\frac{2s}{(s+2)B_1}|ka_*-a^{\frac2{p-1}}|\right)^{\frac1{2(s+1)}},$
then we obtain \eqref{pohdelta}, and  \eqref{pohxb} can be found by  \eqref{pohdelta} and Lemma \ref{lem5.2}.

\end{proof}
\smallskip
\smallskip
\smallskip

By a change of variable, the problem \eqref{eq1}-\eqref{eqnorm} can be rewritten as
\begin{align}\label{eqdelta}
&\delta_a^{2s}(-\Delta)^su+\big(-\mu_a\delta_a^{2s}+\delta_a^{2s}V(x)\big)u=u^p,\ \ u\in H^s(\R^N),
\end{align}
\begin{align}\label{eqnormdelta}
&\int_{\R^N}u^2=(a\delta_a^{2s})^{\frac2{p-1}}.
\end{align}
Similar to Lemma \ref{lemla}, the $k$-peak solution of \eqref{eqdelta}--\eqref{eqnormdelta} concentrating at $b_1,\ldots,b_k$ can be written as
$$u=\sum_{i=1}^k\bar U_{\delta_a,x_{a,i}}+\bar\varphi_a$$
with $|x_{a,i}-b_i|=o(1)$, $\|\bar\varphi_a\|_{\delta_a}=o(\delta_a^{\frac N2})$. Moreover,
$$
\bar\varphi_a\in\bigcap_{i=1}^k\bar E_{a,x_{a,i}}:=\Big\{\phi\in H^s(\R^N):\langle\phi,\frac{\partial\bar U_{\delta_a,x_{a,i}}}{\partial x_j}\rangle_{\delta_{a}}=0, j=1,\ldots,N\Big\},
$$
where
$$\bar U_{\delta_a,x_{a,i}}(x):
=(1+(\gamma_1+V_i)\delta_a^{2s})^{\frac1{p-1}}U\Big(\frac{(1+(\gamma_1+V_i)\delta_a^{2s})^{\frac1{2s}}(x-x_{a,i})}{\delta_a}\Big),$$
$$\|\phi\|_{\delta_a}^2=\int_{\R^N}(\delta_a^2|(-\Delta)^{\frac s2}\phi|^2+\phi^2).$$
Now, we write the equation \eqref{eqdelta} as
$$\bar{L}_a(\bar\varphi_a)=N_a(\bar\varphi_a)+\bar l_a(x),$$
where
$$
\bar{L}_a(\bar\varphi_a)
=\delta_{a}^{2s}(-\Delta)^{s}\bar\varphi_a
+\big[1+(-\mu_{a}\delta^{2s}_{a}-1)+\delta_{a}^{2s}V\big]\bar\varphi_a
-p\sum_{i=1}^{k}\bar{U}^{p-1}_{\delta_{a},x_{a,i}}\bar\varphi_a,
$$
$N_a,L_a$ are defined by \eqref{Na},\eqref{La}, and
\begin{align*}
\bar l_a(x)=\big(\mu_a\delta_a^{2s}+1+\gamma_{1}\delta^{2s}_{a}-\delta_a^{2s}V(x)+V_{i}\delta^{2s}_{a}\big)\sum_{i=1}^k\bar U_{\delta_a,x_{a,i}}+\Big(\sum_{i=1}^k\bar U_{\delta_a,x_{a,i}}\Big)^p-\sum_{i=1}^k\bar U_{\delta_a,x_{a,i}}^p.
\end{align*}

\begin{Lem}
It holds that
\begin{align}\label{barphi}
\|\bar \varphi_a\|_{\delta_a}=O\big(\delta_a^{\frac N2+2s+2}\big).
\end{align}

\end{Lem}

\begin{proof}
Similar to Lemma \ref{lemla}, we obtain \eqref{barphi} only from proving
\begin{align*}
\|\bar l_a\|_{\delta_a}=O\Big(\sum_{i=1}^k|V(x_{a,i})-V_i|\delta_a^{\frac N2+2s}+|\sum_{i=1}^k\nabla V(x_{a,i})|\delta_a^{\frac N2+2s+1}+\delta_a^{\frac N2+2s+2}\Big)
=O\big(\delta_a^{\frac N2+2s+2}\big).
\end{align*}
\end{proof}

\smallskip\smallskip

Let $u_a^{(1)}$ and $u_a^{(2)}$ be two $k$-peak solutions to \eqref{eqdelta} and \eqref{eqnormdelta} concentrating at k points $b_1,\ldots,b_k$, of the form
\begin{align}\label{5.26}
u_a^{(l)}=\sum_{i=1}^k\bar U_{\delta_a,x_{a,i}^{(l)}}+\bar\varphi_a^{(l)},\ \ \bar\varphi_a^{(l)}\in \bigcap_{i=1}^k\bar E_{a,x_{a,i}^{(l)}}\ \ l=1,2.
\end{align}

Set $\xi_{a}(x)=\frac{u_{a}^{(1)}(x)-u_{a}^{(2)}(x)}{\|u_{a}^{(1)}-u_{a}^{(2)}\|_{L^\infty(\R^N)}}$. Then $\|\xi_a\|_{L^\infty(\R^N)}=1$.
We get from \eqref{eqdelta} that
\begin{align*}
\delta^{2s}(-\Delta)^s\xi_a+C_a(x)\xi_a=g_a(x),
\end{align*}
where
\begin{align*}
&C_a(x)=\delta^{2s}V(x)-\delta^{2s}\mu_a^{(1)}-p\int_0^1(tu_a^{(1)}+(1-t)u_a^{(2)})^{p-1}dt,\nonumber\\
&g_a(x)=\frac{\delta_a^{2s}(\mu_a^{(1)}-\mu_a^{(2)})}{\|u_{a}^{(1)}-u_{a}^{(2)}\|_{L^\infty(\R^N)}}u_{a}^{(2)}.
\end{align*}

Following \cite{davilapinowei} and \eqref{decay}, we can prove that

\begin{align}\label{xidecay}
|\xi_a(x)|+|\nabla\xi_a(x)|\leq C\sum_{i=1}^k\frac{1}{\left(1+|\frac{x-x_{a,i}}{\delta_a}|^2\right)^{\frac{N+2s-\theta}{2}}}.
\end{align}

Applying blowing-up technique, we set $$\bar \xi_{a,i}(x)=\xi_{a,i}\Big(\frac{\delta_a}{P_i^{\frac1{2s}}}x+x_{a,i}^{(1)}\Big),\ \ P_i=1+(\gamma_1+V_i)\delta^{2s}.$$
The $\bar\xi_{a,i}$ satisfies that
\begin{align}\label{barxi}
(-\Delta)^s\bar\xi_{a,i}+\frac{1}{P_i}C_a(\frac{\delta_a}{P_i^{\frac1{2s}}}x+x_{a,i}^{(1)})\bar\xi_{a,i}=\frac{1}{P_i}g_a(\frac{\delta_a}{P_i^{\frac1{2s}}}x+x_{a,i}^{(1)}).
\end{align}

\begin{Lem}\label{lemcaga}
For $x\in B_{\delta_a^{-1}P_i^{\frac1{2s}}\rho}(0)$, it holds that
\begin{align}\label{Ca}
\frac{1}{P_i}C_a(\frac{\delta_a}{P_i^{\frac1{2s}}}x+x_{a,i}^{(1)})=1-pU^{p-1}(x)
+O\Big(\delta_a^{2s}+\sum_{l=1}^2\varphi_a^{(l)}(\frac{\delta_a}{P_i^{\frac1{2s}}}x+x_{a,i}^{(1)})\Big),
\end{align}
and
\begin{align}\label{ga}
\frac{1}{P_i}g_a(\frac{\delta_a}{P_i^{\frac1{2s}}}x+x_{a,i}^{(1)})=-\frac{p-1}{ka_*}U(x)\sum_{j=1}^k
\int_{\R^N}U^p(x)\bar\xi_{a,j}(x)+O\Big(\delta_a^{2s+2}+\sum_{l=1}^2\varphi_a^{(l)}(\frac{\delta_a}{P_i^{\frac1{2s}}}x+x_{a,i}^{(1)})\Big).
\end{align}

\end{Lem}

\begin{proof}
Since
\begin{align*}
C_a(\frac{\delta_a}{P_i^{\frac1{2s}}}x+x_{a,i}^{(1)})
=&\delta^{2s}V(\frac{\delta_a}{P_i^{\frac1{2s}}}x+x_{a,i}^{(1)})-\delta^{2s}\mu_a^{(1)}\nonumber\\&
-p\int_0^1(tu_a^{(1)}+(1-t)u_a^{(2)})^{p-1}(\frac{\delta_a}{P_i^{\frac1{2s}}}x+x_{a,i}^{(1)})dt,
\end{align*}
we can obtain \eqref{Ca} directly.

From \eqref{eqdelta} and \eqref{eqnormdelta}, for $l=1,2$, we get that
\begin{align*}
\mu_a^{(l)}\delta_a^{\frac{4s}{p-1}+2s}a^{\frac2{p-1}}=\delta_a^{2s}\int_{\R^N}
\Big(|(-\Delta)^{\frac s2}u_a^{(l)}|^2+V(x)(u_a^{(l)})^2\Big)-\int_{\R^N}|u_a^{(l)}|^{p+1},
\end{align*}
which implies that
\begin{align}\label{5.34}
&\frac{\delta_a^{\frac{4s}{p-1}+2s}a^{\frac2{p-1}}(\mu_a^{(1)}-\mu_a^{(2)})}{\|u_{a}^{(1)}-u_{a}^{(2)}\|_{L^\infty(\R^N)}}\nonumber\\
=&\delta_a^{2s}\int_{\R^N}((-\Delta)^{\frac s2}u_a^{(1)}+(-\Delta)^{\frac s2}u_a^{(2)})\cdot(-\Delta)^{\frac s2}\xi_a
+V(x)(u_a^{(1)}+u_a^{(2)})\xi_a\nonumber\\&
-\int_{\R^N}\frac{(|u_a^{(1)}|^{p+1}-|u_a^{(2)}|^{p+1})}{\|u^{(1)}_{a}-u^{(2)}_{a}\|_{L^{\infty}(\R^{N})}}\nonumber\\
=&\delta_a^{2s}\int_{\R^N}((-\Delta)^{\frac s2}u_a^{(1)}+(-\Delta)^{\frac s2}u_a^{(2)})\cdot(-\Delta)^{\frac s2}\xi_a
+V(x)(u_a^{(1)}+u_a^{(2)})\xi_a\nonumber\\&
-\int_{\R^N}\frac{|u_a^{(1)}|^{p+1}-|u_a^{(2)}|^{p+1}}{\|u_{a}^{(1)}-u_{a}^{(2)}\|_{L^\infty(\R^N)}}-\mu_a^{(2)}\delta_a^{2s}\int_{\R^N}(u_a^{(1)}+u_a^{(2)})\xi_a\nonumber\\
=&(\mu_a^{(1)}-\mu_a^{(2)})\delta^{2s}_{a}\int_{\R^N}u_a^{(1)}\xi_a+\int_{\R^N}(|u_a^{(1)}|^{p}+|u_a^{(2)}|^{p})\xi_a-
(p+1)\int_{\R^N}\xi_a\int_0^1(tu_a^{(1)}+(1-t)u_a^{(2)})^{p}dt.
\end{align}
Hence, in view of \eqref{pohdelta}, \eqref{pohxb}, \eqref{barphi}, \eqref{5.34} gives that
\begin{align*}
\frac{\delta_a^{2s}(\mu_a^{(1)}-\mu_a^{(2)})}{\|u_{a}^{(1)}-u_{a}^{(2)}\|_{L^\infty(\R^N)}}
=&-\frac{p-1}{ka_*\delta_a^{N}}\Big(\int_{\R^N}(|u_a^{(1)}|^{p}+|u_a^{(2)}|^{p})\xi_a-
(p+1)\xi_a\int_0^1(tu_a^{(1)}+(1-t)u_a^{(2)})^{p}dt\Big)\nonumber\\
&+O(\delta_a^{2s+2}),
\end{align*}
which implies \eqref{ga}.
\end{proof}

\medskip

Since $|\bar\xi_{a,i}|\leq1$, we suppose that $\bar\xi_{a,i}\rightarrow\bar \xi_i(x)$ in $C^1_{loc}(\R^N)$.
Lemma \ref{lemcaga} actually shows the following result.

\begin{Lem}
There hold that
\begin{align}\label{eqbarxi}
(-\Delta)^s\bar\xi_i+(1-pU^{p-1})\bar\xi_i=-\frac{p-1}{ka_*}U(x)\sum_{j=1}^k
\int_{\R^N}U^p(x)\bar\xi_{j}(x),\ \ i=1,\ldots,k.
\end{align}

\end{Lem}

In order to show $\bar{\xi}_i=0$, we write
\begin{align}\label{barzeta}
\bar\xi_{a,i}(x)=\sum_{j=0}^N\gamma_{a,i,j}\Psi_j+\bar\zeta_{a,i}(x),\ \ in\ H^s(\R^N),
\end{align}
where $\Psi_j(j=0,1,\ldots,N)$ are functions in \eqref{Psi}, $\bar\zeta_{a,i}\in\bar E$ and
$$\bar E=\{u\in  H^s(\R^N), \langle u,\Psi_j\rangle=0,\ for\ j=0,1,\ldots,N\}.$$

It is standard to show the following problem.
\begin{Lem}\label{lemcoe}
For any $u\in\bar E$, there holds that
\begin{align*}
\|\bar L(u)\|\geq c_0\|u\|,
\end{align*}
where $c_0>0$ is a constant, and the linear operator is defined by
\begin{align}\label{eqxii}
\bar L(u)=(-\Delta)^su+(1-pU^{p-1})u+\frac{p-1}{a_*}U(x)
\int_{\R^N}U^p(x)u(x)dx.
\end{align}
\end{Lem}

\begin{Prop}
The $\bar\xi_{a,i}(x)$ defined by \eqref{barzeta} satisfies that     
\begin{align}\label{5.39}
\|\bar\zeta_{a,i}\|=O(\delta_a^{2s}).
\end{align}
\end{Prop}

\begin{proof}
Using \eqref{barxi}-\eqref{eqbarxi}, for any $\psi\in H^{s}(\R^{N})$ we have by Lemma \ref{lemcaga} and \eqref{barphi} that
\begin{align*}
\langle\bar L(\bar\zeta_{a,i}),\psi\rangle&=\int_{\R^{N}}\big[(1-pU^{p-1}-\frac{1}{P_i}C_a(\frac{\delta_a}{P_i^{\frac1{2s}}}x+x_{a,i}^{(1)}))\bar\xi_{a,i}\big]
\psi\nonumber\\&
\quad+\int_{\R^{N}}\Big[\frac{1}{P_i}g_a(\frac{\delta_a}{P_i^{\frac1{2s}}}x+x_{a,i}^{(1)})+\frac{p-1}{ka_*}U(x)\sum_{j=1}^k
\int_{\R^N}U^p(x)\bar\xi_{a,j}(x)\Big]\psi\nonumber\\&
=O(\delta_a^{2s})\|\psi\|+O\Big(\int_{\R^{N}}\sum_{l=1}^2\varphi_a^{(l)}
(\frac{\delta_a}{P_i^{\frac1{2s}}}x+x_{a,i}^{(1)}))|\psi|\Big)\nonumber\\
&=O(\delta_a^{2s})\|\psi\|+O(\delta_a^{-\frac{N}{2}})\|\varphi_a^{(l)}
\|_{\delta_{a}}\|\psi\|=O(\delta_a^{2s})\|\psi\|,
\end{align*}
which implies \eqref{5.39} by Lemma \ref{lemcoe}.

\end{proof}

\smallskip\smallskip

In the following we prove $$\gamma_{a,i,j}=o(1)~~\mbox{for}~~ j=0,1,\ldots,N.$$
First, we have the estimate for $\xi_a$.
\begin{Lem}
It holds that
\begin{align}\label{5.41}
\begin{split}
&\delta_a^{2s}\sum_{i=1}^k\int_{B_\rho(x_{a,i}^{(1)})}\left(sV(x)+\frac12\langle y-x_{a,i}^{(1)},\nabla V(x)\rangle
\right)(u_a^{(1)}+u_a^{(2)})\xi_a
\\
=&(\mu_a^{(1)}-\mu_a^{(2)})\delta_a^{2s}\int_{\R^N}u_a^{(1)}\xi_a+\int_{\R^N}((u_a^{(1)})^p+(u_a^{(2)})^p)\xi_a\\&
-(p+1)\int_{\R^N}\int_0^1(tu_a^{(1)}+(1-t)u_a^{(2)})^{p}\xi_a\\&
+(\frac N{p+1}-\frac{N-2s}2)(p+1)\int_{B_\rho(x_{a,i}^{(1)})}\int_0^1(tu_a^{(1)}+(1-t)u_a^{(2)})^{p}\xi_a+O(\delta_a^{2N+2s}).
\end{split}
\end{align}
\end{Lem}

\begin{proof}
Since $u_a^{(l)}$ are the $k$-peak solutions of \eqref{eqdelta}--\eqref{eqnormdelta},              
we have the following Pohozaev identity:
\begin{align*}
\begin{split}
&-\delta_a^{2s}\int_{\partial'' {\bf B}_\rho(x_{a,i}^{(1)})}t^{1-2s}\frac{\partial\tilde u_a^{(l)}}{\partial\nu}\langle Y-X_{a,i},\nabla\tilde u_a^{(l)}\rangle
+\frac{\delta_a^{2s}}2\int_{\partial'' {\bf B}_\rho(x_{a,i}^{(1)})}t^{1-2s}|\nabla\tilde u_a^{(l)}|^2\langle Y-X_{a,i},\nu\rangle
\\
&-\frac{N-2s}2\delta_a^{2s}\int_{\partial'' {\bf B}_\rho(x_{a,i}^{(1)})}t^{1-2s}\frac{\partial\tilde u_a^{(l)}}{\partial\nu}\tilde u_a^{(l)}\\&
-\int_{\partial B_\rho(x_{a,i}^{(1)})}\Big(\frac{(u_a^{(l)})^{p+1}}{p+1}-\frac{\delta_a^{2s}(-\mu_a^{(l)}+V(x))}{2}(u_a^{(l)})^{2}
\Big)\langle y-x_{a,i},\nu\rangle
\\
 =&\int_{B_\rho(x_{a,i}^{(1)})}(\frac{N-2s}2-\frac N{p+1})(u_a^{(l)})^{p+1}+s\delta_a^{2s}(-\mu_a^{(l)}+V(x))(u_a^{(l)})^{2}\\&
 +\frac{\delta_a^{2s}}2\langle y-x_{a,i}^{(1)},\nabla V(x)\rangle
 (u_a^{(l)})^2,
\end{split}
\end{align*}
which can be rewritten as
\begin{align*}
\begin{split}
 &\delta_a^{2s}\int_{B_\rho(x_{a,i}^{(1)})}\big(sV(x)(u_a^{(l)})^{2}
 +\frac{1}2\langle y-x_{a,i}^{(1)},\nabla V(x)\rangle
 (u_a^{(l)})^2\big)\\
 &=\int_{B_\rho(x_{a,i}^{(1)})}\delta^{2s}_{a}s\mu_a^{(l)}(u_a^{(l)})^{2}+(\frac N{p+1}-\frac{N-2s}2)(u_a^{(l)})^{p+1}\\&+\delta_a^{2s}\int_{\partial'' \bf B_\rho(x_{a,i}^{(1)})}W_1^{(l)}
 +\delta_a^{2s}\int_{\partial B_\rho(x_{a,i}^{(1)})}W_2^{(l)},
\end{split}
\end{align*}
where
\begin{align*}
W_1^{(l)}=-t^{1-2s}\frac{\partial\tilde u_a^{(l)}}{\partial\nu}\langle Y-X_{a,i},\nabla\tilde u_a^{(l)}\rangle
+\frac{1}2t^{1-2s}|\nabla\tilde u_a^{(l)}|^2\langle Y-X_{a,i},\nu\rangle
-\frac{N-2s}2t^{1-2s}\frac{\partial\tilde u_a^{(l)}}{\partial\nu}\tilde u_a^{(l)},
\end{align*}
\begin{align*}
W_2^{(l)}=\Big(-\frac{(u_a^{(l)})^{p+1}}{\delta_a^{2s}(p+1)}+\frac{(-\mu_a^{(l)}+V(x))}{2}(u_a^{(l)})^{2}
\Big)\langle y-x_{a,i},\nu\rangle.
\end{align*}

Then, the above local Pohozaev identity implies that
\begin{align}\label{5.46}
\begin{split}
 &\delta_a^{2s}\int_{B_\rho(x_{a,i}^{(1)})}\Big(sV(x)
 +\frac{1}2\langle y-x_{a,i}^{(1)},\nabla V(x)\rangle\Big)
 (u_a^{(l)}+u_a^{(2)})\xi_a
 \\
=&\frac{(\mu_a^{(1)}-\mu_a^{(2)})\delta_a^{2s}}{\|u_a^{(l)}-u_a^{(2)}\|_{L^\infty}}\int_{B_\rho(x_{a,i}^{(1)})}(u_a^{(1)})^2+
\mu_a^{(2)}\delta_a^{2s}\int_{B_\rho(x_{a,i}^{(1)})}(u_a^{(1)}+u_a^{(2)})\xi_a\\&
+(\frac N{p+1}-\frac{N-2s}2)(p+1)\int_{B_\rho(x_{a,i}^{(1)})}\int_0^1(tu_a^{(1)}+(1-t)u_a^{(2)})^{p}\xi_a+J_{a,i,1}+J_{a,i,2},
\end{split}
\end{align}
where
\begin{align*}
\begin{split}
 &J_{a,i,1}=\int_{\partial'' \bf{B}_{\rho}(x_{a,i}^{(1)})}\frac{ \delta_a^{2s}(W_1^{(1)}-W_1^{(2)})}{\|u_a^{(l)}-u_a^{(2)}\|_{L^\infty}}\\
& =-\delta_a^{2s}\int_{\partial'' \bf{B}_{\rho}(x_{a,i}^{(1)})}t^{1-2s}\frac{\partial\tilde u_a^{(1)}}{\partial\nu}\langle Y-X_{a,i},\nabla\tilde\xi_a\rangle
-\delta_a^{2s}\int_{\partial'' \bf{B}_{\rho}(x_{a,i}^{(1)})}t^{1-2s}\frac{\partial\tilde \xi_a}{\partial\nu}\langle Y-X_{a,i},\nabla\tilde u_a^{(2)}\rangle\\&
+\frac{\delta_a^{2s}}2\int_{\partial'' \bf{ B}_{\rho}(x_{a,i}^{(1)})}t^{1-2s}\nabla(\tilde u_a^{(1)}+\tilde u_a^{(2)})\cdot\nabla\tilde \xi_a\langle Y-X_{a,i},\nu\rangle
\\
&-\frac{N-2s}2\delta_a^{2s}\int_{\partial'' \bf{B}_{\rho}(x_{a,i}^{(1)})}t^{1-2s}(\frac{\partial\tilde \xi_a}{\partial\nu}\tilde u_a^{(1)}
-\frac{\partial\tilde u_a^{(2)}}{\partial\nu}\tilde \xi_a)=O(\delta_a^{2s+2N}),
\end{split}
\end{align*}
and since $p+1>2$
\begin{align*}
\begin{split}
 &J_{a,i,2}=\int_{\partial B_\rho(x_{a,i}^{(1)})}\frac{ \delta_a^{2s}(W_2^{(1)}-W_2^{(2)})}{\|u_a^{(l)}-u_a^{(2)}\|_{L^\infty}}=
 O(\delta_a^{2s+2(N+2s-\theta)}+\delta_a^{(p+1)N+2ps})= O(\delta_a^{2s+2(N+2s-\theta)}).
\end{split}
\end{align*}

Summing \eqref{5.46} from $i=1$ to $i=k$ and considering the decay \eqref{decay} and \eqref{xidecay},
\begin{align*}
\begin{split}
 &\sum_{i=1}^k\delta_a^{2s}\int_{B_\rho(x_{a,i}^{(1)})}\Big(sV(x)
 +\frac{1}2\langle y-x_{a,i}^{(1)},\nabla V(x)\rangle\Big)
 (u_a^{(l)}+u_a^{(2)})\xi_a
 \\
=&\frac{(\mu_a^{(1)}-\mu_a^{(2)})\delta_a^{2s}}{\|u_a^{(l)}-u_a^{(2)}\|_{L^\infty}}\int_{\R^N}(u_a^{(1)})^2+
\mu_a^{(2)}\delta_a^{2s}\int_{\R^N}(u_a^{(1)}+u_a^{(2)})\xi_a\\&
+(\frac N{p+1}-\frac{N-2s}2)(p+1)\int_{B_{\rho}(x_{a,i}^{(1)})}\int_0^1(tu_a^{(1)}+(1-t)u_a^{(2)})^{p}\xi_a+O(\delta_a^{2N+2s}),
\end{split}
\end{align*}
which, combined with \eqref{5.34}, gives \eqref{5.41}, concluding the proof.

\end{proof}

In the following, we first show $\gamma_{a,i,0}=o(1)$ for $i=1,\ldots,k$.
\begin{Lem}\label{lemgamma1}
It holds that
\begin{align*}
\gamma_{a,i,0}=o(1), \  \ i=1,\ldots,k.
\end{align*}
\end{Lem}
\begin{proof}
From \eqref{pohxb}, \eqref{barphi}, \eqref{5.26} and
\begin{align*}
u_a^{(2)}(\frac{\delta_a}{P_i^{\frac1{2s}}}x+x_{a,i}^{(1)})
=P_iU(x)+O(\frac{P_i^{\frac1{2s}}|x_{a,i}^{(1)}-x_{a,i}^{(2)}|}{\delta_a})+\bar\varphi_a(\frac{\delta_a}{P_i^{\frac1{2s}}}x+x_{a,i}^{(1)}),
\end{align*}
we obtain that
\begin{align*}
\begin{split}
&\int_{B_\rho(x_{a,i}^{(1)})}\Big(sV(x)
 +\frac{1}2\langle y-x_{a,i}^{(1)},\nabla V(x)\rangle\Big)
 (u_a^{(l)}+u_a^{(2)})\xi_a\\
 =&\int_{B_\rho(x_{a,i}^{(1)})}\Big(s(V(x)-V_i)
 +\frac{1}2\langle y-x_{a,i}^{(1)},\nabla V(x)\rangle\Big)\Big(\sum_{i=1}^k(\bar U_{\delta_a,x_{a,i}^{(1)}}+\bar U_{\delta_a,x_{a,i}^{(2)}})+\bar\varphi_a^{(1)}+\bar\varphi_a^{(2)})
 \Big)\xi_a\\&
 +\int_{B_\rho(x_{a,i}^{(1)})}sV_i (u_a^{(l)}+u_a^{(2)})\xi_a\\
 =&2\int_{B_\rho(x_{a,i}^{(1)})}\Big(s(V(x)-V_i)
 +\frac{1}2\langle y-x_{a,i}^{(1)},\nabla V(x)\rangle\Big)\bar U_{\delta_a,x_{a,i}^{(1)}}
 \xi_a\\&
 +\int_{B_\rho(x_{a,i}^{(1)})}sV_i (u_a^{(l)}+u_a^{(2)})\xi_a+O(\delta_a^{N+2s+2}).
 \end{split}
\end{align*}

From  \eqref{xb}  \eqref{pohxb}, \eqref{barzeta} and \eqref{5.39}, we obtain that
\begin{align*}
\begin{split}
&\int_{B_\rho(x_{a,i}^{(1)})}(V(x)-V_i)\bar U_{\delta_a,x_{a,i}^{(1)}}
 \xi_a\\
 =&\delta_a^N\int_{\R^N}(V(\frac{\delta_a}{P_i^{\frac1{2s}}}x+x_{a,i}^{(1)})-V_i) U(x)(\sum_{j=0}^N\gamma_{a,i,j}\Psi_j)
 +O(\delta_a^{N+2s+2})\\
 =&-\frac B2\Delta V(b_i)\gamma_{a,i,0}\delta_a^{N+2}+O(\delta_a^{N+2s+2}+\delta_a^{N+3}),
 \end{split}
\end{align*}
where $B$ is the constant as defined in \eqref{B}.

\vskip .1cm
Similarly,
\begin{align*}
\begin{split}
&\int_{B_\rho(x_{a,i}^{(1)})}\langle y-x_{a,i}^{(1)},\nabla V(x)\rangle\bar U_{\delta_a,x_{a,i}^{(1)}}
 \xi_a=-\frac B2\Delta V(b_i)\gamma_{a,i,0}\delta_a^{N+2}+O(\delta_a^{N+2s+2}+\delta_a^{N+3}).
 \end{split}
\end{align*}
Therefore, it holds
\begin{align}\label{5.57}
\begin{split}
LHS\ of\ \eqref{5.41}&=-\frac{2s+1}4B\Delta V(b_i)\gamma_{a,i,0}\delta_a^{N+2+2s}+O(\delta_a^{N+6s}+\delta_a^{N+3+2s})\\&
\quad+s\delta_a^{2s}V_i \int_{B_\rho(x_{a,i}^{(1)})}(u_a^{(l)}+u_a^{(2)})\xi_a.
\end{split}
\end{align}
Moreover, we have
\begin{align*}
\begin{split}
&\int_{B_\rho(x_{a,i}^{(1)})}(u_a^{(l)}+u_a^{(2)})\xi_a\\
&=2\gamma_{a,i,0}\delta_a^N\int_{\R^N}U(U+\frac{p-1}{2s}x\cdot\nabla U)+O
(|x_{a,i}^{(1)}-x_{a,i}^{(1)}|\delta_a^{N-1}+\delta_a^{\frac N2}\|\bar\varphi_a^{(2)}\|_{\delta_a})\\
&=2\gamma_{a,i,0}\delta_a^N(1-\frac{N(p-1)}{4s})\int_{\R^N}U^2+O(\delta_a^{N+2s+2}+\delta_a^{N+3}),
\end{split}
\end{align*}
which, combined with \eqref{5.57}, gives that
\begin{align*}
\begin{split}
LHS\ of\ \eqref{5.41}&=-\frac{2s+1}4B\Delta V(b_i)\gamma_{a,i,0}\delta_a^{N+2+2s}+O(\delta_a^{N+6s}+\delta_a^{N+3+2s})\\&
+sV_ia_*\gamma_{a,i,0}(1-\frac{N(p-1)}{4s})\delta_a^{N+2s}.
\end{split}
\end{align*}

On the other hand, for the RHS of \eqref{5.41},
\begin{align*}
\begin{split}
&\int_{\R^N}((u_a^{(1)})^p+(u_a^{(2)})^p)\xi_a
-(p+1)\int_{\R^N}\int_0^1(tu_a^{(1)}+(1-t)u_a^{(2)})^{p}\xi_a\\&
+(\frac N{p+1}-\frac{N-2s}2)(p+1)\int_{B_\rho(x_{a,i}^{(1)})}\int_0^1(tu_a^{(1)}+(1-t)u_a^{(2)})^{p}\xi_a\\
=&\Big(N-(p-1)-\frac{(N-2s)(p+1)}2\Big)\Big(1-\frac N{p+1}\Big)\gamma_{a,i,0}\int_{\R^N}U^{p+1}+O(\delta^{N+2+4s}_{a})\\
=&\Big(N-(p-1)-\frac{(N-2s)(p+1)}2\Big)\Big(1-\frac {N(p-1)}{2s(p+1)}\Big)\frac1{(\frac N{s(p+1)}-\frac N{2s}+1)}a_*\gamma_{a,i,0}
\\=&-(\tilde\gamma +o(1))a_*\gamma_{a,i,0}\delta_a^{N}
\end{split}
\end{align*}
with some constant $\tilde\gamma>0$ since $N\geq 2s+2$ and $p<\frac{N+2s}{N-2s}$.
Moreover, by \eqref{pohdelta},\eqref{pohxb} and \eqref{barphi}, it holds that
\begin{align*}
\begin{split}
&(\mu_a^{(1)}-\mu_a^{(2)})\delta_a^{2s}\int_{\R^N}u_a^{(1)}\xi_a=O(\delta_a^{N+6s}).
\end{split}
\end{align*}

To sum up, we finally obtain that
\begin{align*}
\begin{split}
&-\frac{2s+1}4B\Delta V(b_i)\gamma_{a,i,0}\delta_a^{N+2+2s}+O(\delta_a^{N+6s}+\delta_a^{N+3+2s}+\delta_a^{2N+2s})\\&
+sV_ia_*\gamma_{a,i,0}(1-\frac{N(p-1)}{4s})\delta_a^{N+2s}
=-(\tilde\gamma +o(1))a_*\gamma_{a,i,0}\delta_a^{N},
\end{split}
\end{align*}
which implies $\gamma_{a,i,0}=o(1)$.

\end{proof}

\begin{Lem}\label{lemgamma2}
It holds that $\gamma_{a,i,j}=o(1)$, for $i=1,\ldots,k$ and $j=1,\ldots,N$.

\end{Lem}

\begin{proof}
This Lemma can be proved just following that in \cite{LPY19}, so we only sketch the main steps as follows.

\vskip 0.2cm

 \noindent\textbf{Step 1:} To prove $\gamma_{a,i,N}=O(\delta_a)$ for $i=1,\cdots,k$.

\vskip 0.2cm

On one hand,
using \eqref{poh2}, we deduce
\begin{equation}\label{3-18}
\displaystyle \int_{B_\rho(x_{a,i}^{(1)})}\frac{\partial V(x)}{\partial  \nu_{a,i}}A_a(x)\xi_{a}=
O\big(\delta_a^{2N}\big),
\end{equation}
where $\nu_{a,i}$ is the outward unit vector of $\partial B_\rho(x_{a,i}^{(1)})$ at $x$, $
A_a(x)=\displaystyle\sum^2_{l=1}u_{a}^{(l)}(x)$.

On the other hand, by \eqref{pohxb}, and the Taylor expansions
\begin{equation}\label{3--18}
\begin{split}
 \displaystyle \int_{B_\rho(x_{a,i}^{(1)})}&\frac{\partial V(x)}{\partial  \nu_{a,i}}A_a(x)\xi_{a} \\=&
\int_{\R^N}  \frac{\partial V(x_{a,i}^{(1)})}{\partial  \nu_{a,i}} A_a(x)\xi_{a}+\int_{\R^N} \Bigl\langle \nabla \frac{\partial V(x_{a,i}^{(1)})}{\partial  \nu_{a,i}}, x-x_{a,i}^{(1)}\Bigr\rangle A_a(x)\xi_{a}
 +O\big(\delta_a^{N+2}\big)\\=&
-\frac{\partial^2 V(x_{a,i}^{(1)})}{\partial  \nu^2_{a,i}} a_*\gamma_{a,i,N} \delta_a^{N+1}+O\big(\delta_a^{N+2}\big).
\end{split}
\end{equation}
Then \eqref{3-18} and \eqref{3--18} imply  $\gamma_{a,i,N}=O(\delta_a)$.

\medskip

\noindent\textbf{Step 2:} To prove  $\gamma_{a,i,j}=o(1)$ for $i=1,\cdots,k$ and $j=1,\cdots,N-1$.

\smallskip

Similar to \eqref{3-18}, we first have
\begin{equation}\label{3.-14}
 \int_{B_\rho(x_{a,i}^{(1)})} \frac{\partial V(x)}{\partial \tau_{a,i,j}}A_a(x)\xi_{a}=O\big(\delta_a^{2N}\big),
 ~\mbox{for}~i=1,\cdots,k,~~j=1,\cdots,N-1.
 \end{equation}
Using suitable rotation, we assume that $\tau_{a,i,1} =(1, 0,\cdots, 0), \cdots$, $\tau_{a,i, N-1} =(0,\cdots,0, 1,  0)$ and $\nu_{a,i} =(0,\cdots,0,1)$.
Under the condition \textup{($V$)}, we know
\begin{equation}\label{1-16-8}
\begin{split}
 \frac{\partial V(\delta_a x+ x^{(1)}_{a,i})}{\partial \tau_{a,i,j}}=&
\delta_a\sum^N_{l=1}\frac{\partial^2 V(x^{(1)}_{a,i})}{\partial x_l \partial \tau_{a,i,j} } x_l+\frac{\delta_a^2}{2}
\sum^N_{m=1}\sum^N_{l=1}\frac{\partial^3 V(x^{(1)}_{a,i})}{\partial x_m  \partial x_l \partial \tau_{a,i,j}} x_m x_l\\&+
\frac{\delta^3_a}{6}\sum^N_{s=1}\sum^N_{m=1}
\sum^N_{l=1}\frac{\partial^4 V(x^{(1)}_{a,i})}{\partial x_s \partial x_l \partial x_m \partial \tau_{a,i,j}} x_s x_lx_m
+o\big(\delta_a^3|x|^3\big),~\mbox{in}~B_{d{\delta_a^{-1}}}(0).
\end{split}\end{equation}
By \eqref{necess}, \eqref{xb}, \eqref{barphi}, \eqref{barzeta},  and  the symmetry of $\varphi_j(x)$, we find, for $j=1,\cdots,N-1$,
\begin{equation*}
\begin{split}
\sum^N_{m=1}\sum^N_{l=1}&\frac{\partial^3 V(x^{(1)}_{a,i})}{\partial x_m  \partial x_l \partial \tau_{a,i,j}}\int_{B_{d\delta_a^{-1}}(0)} A_a(\delta_a x+ x^{(1)}_{a,i}) \bar \xi_{a,i} x_m x_l\\
=& B\gamma_{a,i,0} \frac{\partial \Delta V(x^{(1)}_{a,i})}{\partial \tau_{a,i,j}} +O(\delta_a^{2s})=
O\big(|x^{(1)}_{a,i}-b_i|\big) +O(\delta_a^{2s})=O(\delta_a^{2s}).
\end{split}
\end{equation*}
Also from Step 1, we get
\begin{equation*}
\begin{split}
& \sum^N_{s=1}\sum^N_{m=1}\sum^N_{l=1}\frac{\partial^4 V(x^{(1)}_{a,i})}{\partial x_s \partial x_l \partial x_m \partial \tau_{a,i,j}}
 \int_{B_{d\delta_a^{-1}}(0)} A_a(\delta_a x+ x^{(1)}_{a,i}) \bar \xi_{a,i}x_s x_lx_m
\\=&
-3B\Big(\sum_{q=1}^{N-1}\frac{\partial^2 \Delta V(b_i)}{\partial \tau_{a, i,q} \partial \tau_{a,i, j}}  \gamma_{a,i,q}\Big)+o\big(1\big).
\end{split}
\end{equation*}
By \eqref{ab7-19-29},  we estimate
\[
\frac{\partial^2 V(x^{(1)}_{a,i})}{\partial x_l \partial \tau_{a,i, j} }=  - \frac{\partial V(x^{(1)}_{a,i})}{ \partial \nu_{a,i} }
\kappa_{i,l}(x^{(1)}_{a,i})\delta_{lj},\quad l,j=1,\cdots,N-1.
\]
 Since  $\frac{\partial V(\bar x^{(1)}_{a,i})}{ \partial \nu_{a,i} }=0$, from \eqref{xb}, we find
\begin{equation}\label{lll}
\begin{split}
\frac{\partial^2 V(x^{(1)}_{a,i})}{\partial x_l \partial \tau_{a,i,j} }=\frac{B}{2a_*}\frac{\partial  \Delta V(b_i)}{\partial \nu_i}  \delta_a^{2}\kappa_{i,l}(b_i)\delta_{lj}+o(\delta_a^2).
\end{split}
\end{equation}
Therefore from \eqref{barzeta}, \eqref{barphi} and \eqref{lll}, we get
\begin{equation}\label{5-16-8}
\begin{split}
\sum^N_{l=1} &\frac{\partial^2 V(x^{(1)}_{a,i})}{\partial x_l \partial \tau_{a,i,j} }\int_{B_{d\delta_a^{-1}}(0)} A_a(\delta_a x+ x^{(1)}_{a,i}) \bar \xi_{a,i} x_l
= -\frac{B}{2}\frac{\partial  \Delta V(b_i)}{\partial \nu_i}  \delta_a^{2}\kappa_{i,j}(b_i)\gamma_{a,i,j}
+o(\delta_a^2).
\end{split}\end{equation}
Combining \eqref{1-16-8}--\eqref{5-16-8}, we obtain
\begin{equation}\label{6-16-8}
\begin{split}
\int_{B_\rho(x_{a,i}^{(1)})} & \frac{\partial V(x)}{\partial \tau_{a,i,j}}A_a(x)\xi_{a}\\=&-\frac{B}{2}\frac{\partial  \Delta V(b_i)}{\partial \nu_i}  \delta_a^{N+3}\kappa_{i,j}(b_i)\gamma_{a,i,j}
-\frac{B}{2}\Big(\sum_{l=1}^{N-1}\frac{\partial^2 \Delta V(b_i)}{\partial \tau_{i,l} \partial \tau_{ i,j}} \gamma_{a,i,j} \Big) \delta_a^{N+3}+o(\delta_a^{N+3}).
 \end{split}\end{equation}
 From \eqref{3.-14} and \eqref{6-16-8}, we find
\[
\frac{\partial  \Delta V(b_i)}{\partial \nu_i} \kappa_{i,j}(b_i)\gamma_{a,i,j}
+\Big(\sum_{l=1}^{N-1}\frac{\partial^2 \Delta V(b_i)}{\partial \tau_{i,l} \partial \tau_{i,j}} \gamma_{a,i,l}\Big) =o(1),
\]
which implies $\gamma_{a,i,j}=o(1)$ for $i=1,\cdots,k$ and $j=1,\cdots,N-1$.
\end{proof}

Now we are ready to prove Theorem \ref{th4}.
\begin{proof}[\textbf{Proof of Theorem \ref{th4}:}]
First, for large fixed $R$,  \eqref{Ca} and \eqref{ga} give
$$
C_{a}(x) \ge \frac12,   \;\; |g_a(x)|\le C \sum^k_{i=1}\frac{\delta_a^{N+2s}}{|x-x^{(1)}_{a,i}|^{N+2s}},\quad x\in \mathbb R^N \backslash  \bigcup ^k_{i=1}B_{R\delta_a}(x^{(1)}_{a,i}).$$
  Using the comparison principle associated to the nonlocal operator $(-\Delta)^s$, we get
\begin{equation*}
 \xi_a(x)= o(1),~\mbox{in}~ \R^N\backslash  \bigcup ^k_{i=1}B_{R\delta_a}(x^{(1)}_{a,i}).
\end{equation*}

On the other hand, it follows from \eqref{barzeta}, \eqref{5.39}, Lemmas \ref{lemgamma1} and \ref{lemgamma2} that
\[
\xi_{a}(x)=o(1),~\mbox{in}~ \bigcup ^k_{i=1}B_{R\delta_a}(x^{(1)}_{a,i}).
\]
This is in contradiction with $\|\xi_{a}\|_{L^{\infty}(\R^N)}=1$.
So $u^{(1)}_{a}(x)\equiv u^{(2)}_{a}(x)$ for $a\rightarrow a_0$.
\end{proof}

\section*{Appendix}

\appendix

\section{Basic estimates }

\renewcommand{\theequation}{A.\arabic{equation}}

\begin{Lem}[c.f. \cite{wy10}]\label{lemA1}
Let $\alpha,\beta\geq1$ be two constants. For any $0<\delta\leq\min\{\alpha,\beta\}$, and $x_1\neq x_2$, it holds that
\begin{align*}
&\frac1{(1+|y-x_1|)^\alpha(1+|y-x_2|)^\beta}\leq
\frac C{|x_1-x_2|^\delta}\Big(\frac1{(1+|y-x_1|)^{\alpha+\beta-\delta}}+\frac1{(1+|y-x_2|)^{\alpha+\beta-\delta}}\Big).
\end{align*}
   \end{Lem}

 \begin{Lem}[c.f. \cite{gnnt}]\label{lemA3}
Let $\rho>\theta>0$ be two constants. Suppose $(y-x)^2+t^2\geq\rho^2,t>0$ and $\alpha>N$. Then, when $0<\beta<N$, it holds that
\begin{align}\label{A1}
&\int_{\R^N}\frac1{(t+|z|)^\alpha|y-z-x|^\beta}\leq C\left(
\frac1{(1+|y-x|)^{\beta}}\frac1{t^{\alpha-N}}+\frac1{(1+|y-x|)^{\alpha+\beta-N}}
\right),\end{align}
and
\begin{align*}
&\int_{\R^N\setminus B_\theta(0)}\frac1{(t+|z|)^\alpha|y-z-x|^\beta}\leq
\frac C{(1+|y-x|)^{\beta}}\frac1{\theta^{\alpha-N}};\end{align*}
when $N<\beta$, we have that
\begin{align*}
&\int_{\R^N\setminus B_\theta(y-x)}\frac1{(t+|z|)^\alpha|y-z-x|^\beta}\leq C\left(
\frac1{(1+|y-x|)^{\beta}}\frac1{t^{\alpha-N}}+\frac1{(1+|y-x|)^{\beta}}\frac1{\theta^{\beta-N}}
\right),
\end{align*}
where $C>0$ is a constant independent of $\theta$.
\end{Lem}

\begin{Rem}\label{remA2}
For some $\epsilon\rightarrow0$,
in the case of $\beta<N$, \eqref{A1} implies directly that
\begin{align*}
\begin{split}
 \int_{\R^N}\frac1{(t+|z|)^\alpha(1+\frac{|y-z-x|}{\epsilon})^\beta}&\leq  C\epsilon^\beta \int_{\R^N}\frac1{(t+|z|)^\alpha|y-z-x|^\beta}\\
&\leq C\epsilon^\beta\left(
\frac{1}{(1+|y-x|)^{\beta}}\frac1{t^{\alpha-N}}+\frac{1}{(1+|y-x|)^{\alpha}}
\right).
\end{split}
\end{align*}

While in the case of $\beta>N$, following the proof in   \cite{gnnt}, we can get indeed that,
\begin{align*}
\begin{split}
 \int_{\R^N}\frac1{(t+|z|)^\alpha(1+\frac{|y-z-x|}{\epsilon})^\beta}&\leq C\left(
\frac{\epsilon^\beta}{(1+|y-x|)^{\beta}}\frac1{t^{\alpha-N}}+\frac{\epsilon^N}{(1+|y-x|)^{\alpha}}
\right)\\
&\leq C\epsilon^N\left(
\frac{1}{(1+|y-x|)^{\beta}}\frac1{t^{\alpha-N}}+\frac{1}{(1+|y-x|)^{\alpha}}
\right).
\end{split}
\end{align*}
\end{Rem}

\medskip
Now let $\Gamma\in C^2$ be a closed hypersurface in $\mathbb R^N$. For $y\in \Gamma$, let $\nu(y)$ and $T(y)$ denote respectively the outward unit  normal to $\Gamma$ at $y$ and the tangent hyperplane to $\Gamma$ at $y$. The curvatures of $\Gamma$ at a fixed point $y_0\in\Gamma$ are determined as follows. By a rotation of coordinates, we can assume that
$y_0=0$ and $\nu(0)$ is the $x_N$-direction, and $x_j$-direction is the $j$-th principal direction.
In some neighborhood $\mathcal{N}=\mathcal{N}(0)$ of $0$,   we have
\[
\Gamma=\bigl\{  x:   x_N=\varphi(x')\bigr\},
\]
where $x'=(x_1,\cdots,x_{N-1})$,
 \[
 \varphi(x') =\frac12 \sum_{j=1}^{N-1} \kappa_j x_j^2 +  O\big(|x'|^3\big),
 \]
 where
 $\kappa_j$, is  the $j$-th principal curvature of  $\Gamma$ at $0$.
 The Hessian matrix $[D^2 \varphi(0)]$  is given by
\begin{equation*}
[D^2 \varphi(0)]=diag [\kappa_1,\cdots,\kappa_{N-1}].
\end{equation*}
Suppose that  $W$ is a smooth function, such  that  $W(x)=a$  for all $x\in \Gamma$.

\begin{Lem}[c.f. \cite{LPY19}]
We have
 \begin{equation*}
 \frac{\partial W(x)}{\partial x_l}\Bigr|_{x=0}=0, ~\mbox  ~l=1,\cdots,N-1,
 \end{equation*}
  \begin{equation}\label{ab7-19-29}
  \frac{\partial^2 W(x)}{\partial x_m\partial x_l }\Bigr|_{x=0}=-\frac{\partial W\big(x\big)}{\partial x_N}\Bigr|_{x=0}\kappa_i \delta_{ml}, ~\mbox{for}~m, l=1,\cdots, N-1,
 \end{equation}
 where  $\kappa_1,\cdots,\kappa_{N-1}$, are  the principal curvatures of  $\Gamma$ at $0$.
\end{Lem}

\section{Linearization}

\begin{Lem}
Let $\xi=(\xi_1,\xi_2,\ldots,\xi_k)$ be the solution of the following problem
\begin{align}\label{eqxii}
(-\Delta)^s\xi_i+(1-pU^{p-1})\xi_i=-\frac{p-1}{ka_*}U(x)\sum_{j=1}^k
\int_{\R^N}U^p(x)\xi_{j}(x),\ \ i=1,\ldots,k.
\end{align}
Then, there holds that
\begin{align}\label{A2}
\xi_i(x)=\sum_{j=0}^N\gamma_{i,j}\Psi_j,
\end{align}
where $\gamma_{i,j}$ are constants, and
\begin{align}\label{Psi}
\Psi_0=U+\frac{p-1}{2s} x\cdot\nabla U,\ \ \Psi_j=\frac{\partial U}{\partial x_j},\ j=1,\ldots,N.
\end{align}
Moreover, $\gamma_{i,0}=\gamma_{l,0}$ for $i,l=1,\ldots,k$.
\end{Lem}

\begin{proof}

Setting $$\bar L(u)=(-\Delta)^su+(1-pU^{p-1})u~~\mbox{and}~~\Psi^j=(\Psi_j,\ldots,\Psi_j),$$ then it is easy to find that $\bar L(\Psi^j)=0$ and $\Psi^j$ solve \eqref{eqxii} for $j=1,\ldots,N$.

 On the other hand, we denote $$\Psi=(U+\frac{p-1}{2s}x\cdot\nabla U,\ldots,U+\frac{p-1}{2s}x\cdot\nabla U),$$ and then
by checking the extension equation associated with \eqref{eqxii}, we could obtain that $\Psi$ solves \eqref{eqxii} too. Noting that $\Psi,\Psi^j$ are linearly independent,
we obtain \eqref{A2}.

Moreover, putting \eqref{A2} into \eqref{eqxii}, we obtain that
$\gamma_{i,0}=\gamma_{l,0}$ for any $i,l=1,\ldots,k$.

\end{proof}

\medskip

 \noindent\textbf{Acknowledgments}
The authors would like to thank Professor Wenxiong Chen for the helpful discussion with him.
Guo was supported by NSFC grants (No.11771469). Luo and Wang were supported by the Fundamental
Research Funds for the Central Universities(No.KJ02072020-0319). Luo was supported by NSFC grants (No.11701204, No.11831009).
Wang was supported by NSFC grants (No.12071169). Yang was supported by NSFC grants (No.11601194).


\end{document}